\documentclass{amsart}
\usepackage{amssymb,amsthm,amsmath}
\usepackage{graphicx}

\newtheorem{theorem}{Theorem}[section]
\newtheorem{proposition}[theorem]{Proposition}
\newtheorem{lemma}[theorem]{Lemma}
\newtheorem{corollary}[theorem]{Corollary}
\theoremstyle{definition}
\newtheorem{definition}[theorem]{Definition}
\newtheorem{example}[theorem]{Example}

\def\K{\mathbb{K}}
\def\k{\mathbf{k}}
\def\R{\mathbb{R}}
\def\Z{\mathbb{Z}}
\def\T{\mathbb{T}}
\renewcommand\P{\mathbb P}
\def\b{\beta}

\begin{document}

\title{Singular tropical hypersurfaces}
\author{Alicia Dickenstein and Luis F. Tabera}
\date{}

\address{\small \rm Departamento de Matem\'atica, Universidad de Buenos Aires,
C1428EGA Buenos Aires, Argentina}
\email{alidick@dm.uba.ar}
\thanks{AD is partially supported by UBACYT X064, CONICET
 PIP 112-200801-00483 and ANPCyT 2008-0902, Argentina}
\address{\small \rm Departamento de Matem\'aticas, Estad\'istica y
Compu\-taci\'on, Universidad de Can\-ta\-bria,
39071, Santander, Spain}
\email{taberalf@unican.es}
\thanks{LFT is partially supported by the coordinated project
MTM2008-04699-C03-{03} and a Research Postdoctoral Grant, Ministerio de Ciencia
e Innovaci\'on, Spain.}

\begin{abstract}
We study the notion of singular tropical hypersurfaces of any dimension. We
characterize the singular points in terms of tropical Euler derivatives
and we give an algorithm to compute all singular points. We also describe 
non-transversal intersection points of planar tropical curves.
\end{abstract}
\maketitle

\section{Introduction}\label{sec:introduction}

The concept of a singular point of a tropical variety is not well established
yet.  A natural definition is the
following. Let $\mathbb{K}$ be an algebraically closed field of characteristic
$0$ with a valuation $val: \K^* \to \R$. We say that a point $q$ in a tropical
variety $V \subset \R^d$ is singular if there exists a singular algebraic
subvariety of the torus $({\K}^*)^d$, with tropicalization $V$, with a
singular point of valuation $q$ (see Definition~\ref{def:tropical_singularity}).
This definition of singularity in terms of the tropicalization of classical
algebraic varieties has been considered in~\cite{sing-fixed-point} in the
case $d=2$ of planar curves, and indirectly in~\cite{tropical-discriminant}
and~\cite{Tesis-master-Ochse}, in the general hypersurface case. 
Thus, in principle, one should study all the
\emph{preimages} of $V$ under the valuation map to decide whether $V$ is
singular. We present an equivalent formulation when $V$ is a hypersurface
defined by a tropical polynomial with prescribed support $A$ and the residue
field of $\mathbb{K}$ has also characteristic $0$ (but our approach can be
extended if this hypothesis is relaxed). 

Recall that given a finite set $A \subseteq \mathbb{Z}^d$, Gel'fand, Kapranov
and Zelevinsky \cite{GKZ-book} defined and studied the main properties of the
$A$-discriminant $\Delta_A$ associated to the family of hypersurfaces with
support $A$. Let $\nabla_0$ be the variety of Laurent polynomials $F$ with
coefficients in $\K$ and support in $A$ which define a singular hypersurface
$\{F=0\}$ in the torus $(\mathbb{K}^*)^d$. If $\nabla_0$ has codimension one,
then there exists a unique (up to sign) polynomial $\Delta_A\in \mathbb{Z}[a_i|
i\in A]$ such that if $F=\sum_{i\in A} a_ix^i$ has a singular point in
$(\mathbb{K}^*)^d$ then $\Delta_A((a_i)_{i \in A})=0$. This polynomial is called
the {\em $A$-discriminant} and its locus coincides with the dual variety $X_A^*$
of the equivariantly embedded toric variety $X_A$ rationally parametrized by the
monomials with exponents in $A$. If codim$(\nabla_0)>1$ (in which case $X_A$ is
said to be \emph{defective}), the polynomial $\Delta_A$ is defined to be the
constant polynomial $1$, and we call in this case $A$-discriminant the ideal of
$X_A^*$ . The varieties $X_A$ and $X_A^*$, as well as the $A$-discriminant,
are affine invariants of the configuration $A$.

A tropical polynomial $f =\bigoplus_{i\in A} p_i \odot w^i$ with coefficients in
$\R$ defines a singular tropical hypersurface precisely when its vector of
coefficients $p$ lies in the tropicalization $\mathcal{T}(X_A^*)$ of the
$A$-discriminant. The concept of singularities of tropical varieties, as well as
the
concept of tropical tangency, can thus be addressed via the tropicalization of
the
$A$-discriminant described in~\cite{tropical-discriminant}. 
We call $\widetilde{A} \in \Z^{d \times A}$ the integer matrix with columns
$(1,i), i \in A$. Theorem~1.1 in \cite{tropical-discriminant} states that the
tropical discriminant equals the Minkowski sum of the co-Bergman fan ${\mathcal
B}^*(\widetilde{A})$ and the row space of the matrix $\widetilde{A}$. The
co-Bergman fan ${\mathcal B}^*(\widetilde{A})$ is the tropicalization of the
kernel of $\widetilde{A}$, or equivalently, the tropicalization of the space of
affine relations among the vectors $i \in A$. Following the notations in
Section~\ref{sec:Euler_derivatives} (cf. also the  
discussion in Section~\ref{sec:Bergman}), we can give an equivalent appealing
definition of singular point of the tropical hypersurface defined by $f =
\bigoplus_{i\in A} p_i \odot w^i$. Let $\phi(x) = \sum_{i\in A} x_i$ be the
linear form with all  coefficients equal to $1$, and denote by $\Phi$ the
tropical hypersurface $\Phi:= \mathcal{T}(Trop(\phi))$ consisting of those
vectors $v \in \R^A$ for which the minimum of the coordinates of $v$ is attained
at least twice. Clearly, $\mathcal{B}^*(\widehat{A})\subseteq \Phi$. 
We then have
\[ \mathcal{T} (f) =\{ w \in \R^d : w \cdot A + p \in \Phi\}, \] 
and the singularities ${\rm sing}(\mathcal{T}(f)) \subseteq \mathcal{T} (f)$ are
described by
\[ {\rm sing}(\mathcal{T}(f)) = \{ q \in \R^d : q \cdot A + p \in
\mathcal{B}^*(\widehat{A})\}.\]

Deciding whether a given tropical polynomial defines a singular tropical
hypersurface, amounts with this approach to finding a way of writing its vector
of coefficients $p$ as the sum of an element in ${\mathcal B}^*(\widetilde{A})$
plus
an element in the rowspan of $\widetilde{A}$. This is particularly involved when
$X_A$ is defective, and there was no algorithm known in the general case
(cf.~\cite{Tesis-master-Ochse}, where an algorithm is presented under some
geometric assumptions, or the arguments in the proof of
\cite[Lemma~3.12]{sing-fixed-point}). 

Instead, we give in Theorem~\ref{teo:singular_point_by_derivatives} a direct
characterization of tropical singular points in terms of analogs of Euler
derivatives of tropical polynomials, which allows us to recover Theorem~1.1
in~\cite{tropical-discriminant}. Our tropical approach translates into an
algorithm to decide whether the tropical variety associated to $f$ is singular
and to detect all the singular points.

Note that, given a Laurent polynomial $F = \sum_{i\in A} a_i x^i$, if the vector
of valuations $q= (val(a_i))_{i \in A}$ defines a non singular tropical
hypersurface, we get a ``certificate'' that $F$ defines a non singular
hypersurface in the torus $(\K^*)^d$. However, it is not possible to find a
simple combinatorial formula to describe all singular points because the
situation is not, as one could expect, completely local (cf.
Proposition~\ref{prop:anotherpt}, and the concept of $\Delta$-equivalence by
Gel'fand, Kapranov and Zelevinsky). We give  several combinatorial
conditions which characterize singular tropical hypersurfaces.

In \cite{mikha} or \cite{inflection-points}, tropical smooth curves are defined
in terms of coherent subdivisions where all points in $A$ are marked, and which
define primitive triangulations of the convex hull $N(A)$ of $A$ (cf. also the
concept of singular tropical curves of maximal dimensional type in
\cite{sing-fixed-point}). If the dual subdivision of $N(A)$ induced by a
tropical curve ${\mathcal T}(f)$ is a primitive triangulation, ${\mathcal T}(f)$
will always be a smooth curve in our sense too. But our definition allows for
certain non primitive triangulations which correspond to smooth tropical
hypersurfaces, that is, tropical hypersurfaces that cannot be the
tropicalization of an algebraic hypersurface with a singularity in the algebraic
torus $(\mathbb{K}^*)^2$ (see Examples~\ref{ex:triang2},~\ref{ex:nontriang2} and
Proposition~\ref{teo:singular_dim_1_and_2}). When $A$ does not admit a
unimodular triangulation or if not all the lattice points in the convex hull of
$A$ are marked, we can have smooth points at facets where the weight (as it is
currently defined, see \cite{Mik05}) is $> 1$. Note that we concentrate on {\em
affine} singular points, that is points in the ``torus'' $\R$ of the tropical
semifield.

Our study includes all coherent subdivisions of $A$. Thus, we refine the
definition in \cite{mikha} and we explore the whole $A$-discriminant. Our method
also generalizes trivially to hypersurfaces in arbitrary dimension. In the last
section, we apply our tools to define and study the non-transversal
intersections of two tropical curves with fixed monomial support, that is, the
tropicalization of mixed discriminants of bivariate polynomials. 

\section{Tropical singularities through Euler
derivatives}\label{sec:Euler_derivatives}

We fix throughout the text a finite lattice set $A$ in $\Z^d$ of cardinality
$n$. We will assume without loss of generality that the $\Z$-linear span of $A$
equals $\Z^d$. The $\R$-{\em affine} span of a subset $S$ in $\R^d$ will be
denoted by $\langle  S \rangle$.

We consider the tropical semifield $(\T,\oplus, \otimes)$, where $\T = \R \cup
\{\infty\}$ and the tropical operations are defined by $ w \oplus w'=
\min\{w,w'\}$, $ w \odot w'= w + w'$. Our object of study are {\em tropical
polynomials} $f =\bigoplus_{i\in A} p_i \odot w^{\odot i} \in \mathbb{T}[w_1,
\ldots, w_d]$ with {\em support $A$}, that is
$p_i \in \mathbb{R}$ for {\em all} $i$.
To simplify the notation, we will write $w^{\odot j} = w^j = \langle  j, w
\rangle$. 

The {\em tropical hypersurface} defined by a non zero tropical polynomial $f$
with support $A$,  is  the set
{\small
\begin{equation}\label{eq:Vf} \mathcal{T}(f) \, = \,
\{w \in \R^d : \exists i_1 \not= i_2 \in A \text{ such that }
{f(w)} = \langle  i_1, w \rangle + p_{i_1} = \langle  i_2, w \rangle +
p_{i_2}\}. 
\end{equation}
}
Any tropical hypersurface is a rational polyhedral complex. For any $q \in
\mathcal{T}(f)$, its associated cell $\sigma^*$ is the closure of all the points
$q' \in \mathcal{T}(f)$ for which $f(q)$ and $f(q')$ are attained at the same
subset $\sigma$ of $A$. Each cell $\sigma^*$ comes with a marking, given by the
subset $\sigma$. So, a tropical hypersurface associated with a tropical
polynomial with fixed support $A$ will be a {\em marked rational polyhedral
complex}. This marking will be transparent in the notation. We refer to the
beginning of section~\ref{sec:Bergman} for further details.

We will also work with Laurent polynomials $F$ with support $A$ and coefficients
in an algebraically closed field $\mathbb{K}$ of characteristic $0$, that is
\begin{equation} \label{eq:F}
F(x) = \sum_{i\in A} a_i x^i \in \mathbb{K}[x_1^{\pm1}, \ldots, x_d^{\pm 1}].
\end{equation}
We will assume that the field $\K$ is provided with a rank-one non-archimedean
valuation $val: \K \to \R$, and that the residue field $\k$ of $\K$ is also of
characteristic zero.
The {\em tropicalization} of a non zero polynomial $F$ as in \eqref{eq:F} is the
tropical polynomial
\begin{equation}\label{eq:TdeF}
f = Trop(F) = \bigoplus_{i\in A} val(a_i) \odot w^{i}.
\end{equation}
When the valuation group is not the whole of $\R$, we will suppose that the
coefficients $p_i$ of a tropical polynomial $f = \bigoplus_{i\in A} p_i \odot
w^{\odot i}$ or a tropical point $q$ that we want to lift {\em lie in the image}
of the valuation map. To accompany our notions in the classical and tropical
settings, the elements of $\K, \K^d$ and $\K^n$ will be denoted systematically
by the letters $a,b,c$, $x,y,z$ and the elements of $\T, \T^d$ and $\T^n$ by the
letters $p,q,w,v,l$. We will denote by $t$ a fixed element of $\K$ of valuation
one. The elements of $A\subset \Z^d$ with be denoted by the letter $i$.

We introduce now the notion of singular point of a tropical hypersurface.

\begin{definition}\label{def:tropical_singularity}
Let $A\subseteq \mathbb{Z}^d$ as before. Let $f=\bigoplus_{i\in A} p_i \odot w^i
\in \R[w_1,\ldots, w_d]$ be a tropical polynomial. Let $q$ be a point in the
tropical hypersurface $\mathcal{T}(f)$. Then, $q$ is a \emph{singular point} of
$\mathcal{T}(f)$ if there is a polynomial $F=\sum_{i\in A} a_i x^i \in
\K[x_1^{\pm1}, \ldots, x_d^{\pm 1}]$ and a point $b\in (\mathbb{K}^*)^d$ such
that $val(a_i)=p_i$, $val(b)=q$ and ${b}$ is a singular point of the algebraic
hypersurface $V(F)$ defined by $F$. If $\mathcal{T}(f)$ has a singular point, we
call it a \emph{singular tropical hypersurface}.
\end{definition}

For instance, if $A = \{0, \dots, m\} \in \Z$, with $m \ge 2$, and $f =
\bigoplus_{j=0}^m 0 \odot w^j$, then $q=0$ is always a singular point of
$\mathcal{T}(f)$  since for all $m$ there exist univariate polynomials of degree
$m$ with coefficients of valuation $0$ and multiple roots with valuation $0$
(just consider $F = (x-1)^m$ which has a multiple root at $1$). 

Let $L(w)$ be an {\em integral affine function} on $\R^d$,
\begin{equation}\label{eq:L}
L(w)=j_1 w_1+\ldots+ j_d w_d +\b,
\end{equation}
where $(j_1, \dots, j_d) \in \Z^d$ and $\b \in \Z$. The Euler derivative of a
tropical polynomial $f$ with support in $A$ with respect to $L$ is defined as
follows.

\begin{definition}\label{def:trop_derivative}
Let $f=\bigoplus_{i\in A} p_i \odot w^i$ and $L=j_1 w_1+\ldots+ j_d w_d +\b$ be
an integral affine function. The \emph{Euler derivative} of $f$ with respect to
$L$ is the tropical polynomial
\[\frac{\partial f}{\partial L}= \bigoplus_{i\in A, L(i) \not=0} p_i \odot
w^i.\]
\end{definition}

We also have the standard Euler derivative of classical Laurent polynomials. 

\begin{definition}\label{def:euler_derivative}
Let $F=\sum_{i\in A} a_i x^i \in \mathbb{K}[x_1^{\pm1}, \ldots, x_d^{\pm 1}]$
and $L=j_1 w_1+\ldots+ j_d w_d +\b$ be an integral affine function. We associate
to
$L$ the Euler vector field $L_\Theta = j_1 \Theta_1 + \ldots + j_d \Theta_d +
\b$, where $\Theta_j = x_j \frac{\partial}{\partial x_j}$ for all $j=1, \dots,
d$. The \emph{Euler derivative} of $F$ with respect to $L$ is the polynomial
\[\frac{\partial F}{\partial L}:= L_\Theta(F) = j_1 x_1\frac{\partial
F}{\partial x_1}+\ldots + j_dx_d\frac{\partial F}{\partial x_d} +\b F.\]
\end{definition}
It is clear that for any singular point $b \in (\K^*)^d$ of $V(F)$, it holds
that $\frac{\partial F}{\partial L}(b) =0$ for all integral affine functions
$L$. Note that if $L$ is the constant function $1$, then $\frac{\partial
F}{\partial L} = F$.

We relate the derivative of $F$ with respect to $L$ with the derivative with
respect to $L$ of its tropicalization.

\begin{lemma}\label{lem:LFf} Given a tropical polynomial $f$ with support $A$
and an integral affine function $L$, the equality
\[\frac{\partial f}{\partial L} = Trop\left(\frac{\partial F}{\partial
L}\right)\]
holds for any polynomial $F$ with support $A$ such that $Trop(F) = f$.
\end{lemma}

\begin{proof}
Take any $F$ with $Trop(F)=f$. Note that the Euler derivative of $F$ with
respect to $L$ equals $ \frac{\partial F}{\partial L} \, = \, \sum_{i\in A} L(i)
a_i x^i.$ From our assumption that the residue field of $\K$ is also of
characteristic zero, it follows that $val(L(i)) =0$ whenever $L(i) \not=0$ and
$val(L(i)) = \infty$ otherwise. The result is then a direct consequence of
Definition~\ref{def:trop_derivative} of the Euler derivative with respect to $L$
in the tropical context. 
\end{proof}

As $A$ is finite, the set $\left\{\frac{\partial f}{\partial L}| L\right\}$,
with $L$ ranging over all possible integer affine linear functions, is finite
for any $f$ with support $A$.

\begin{example}\label{ex:nsconic}
Consider the tropical conic $f = 1 \oplus 0 \odot w_1 \oplus 0 \odot w_2 \oplus
0 \odot w_1\odot w_2 \oplus 1 \odot w_1^2 \oplus 1 \odot w_2^2$. Let $L_1 = w_1,
L_2 = w_2$ and let $F = a_{(0,0)}+ a_{(1,0)} x + a_{(0,1)} y +$ $a_{(1,1)} xy +
a_{(2,0)} x^2+ a_{(0,2)} y^2$ be any polynomial with tropicalization $f$. The
associated $A$ discriminant $\Delta_A$ equals $1/2$ of the determinant of the
matrix
\begin{equation}\label{eq:matrix}
\begin{pmatrix}
2 a_{(2,0)} & a_{(1,1)} & a_{(1,0)}\\
 a_{(1,1)} & 2 a_{(0,2)} & a_{(0,1)}\\
a_{(1,0)} & a_{(0,1)} & 2 a_{(0,0)}
\end{pmatrix},
\end{equation}
which is non zero since $ 2 a_{(1,0)} a_{(1,1)} a_{(0,1)}$ is the only term in
the expansion of the determinant with lowest valuation $0$. Thus, as one could
expect, $\mathcal{T}(f)$ is a non singular tropical hypersurface according to
Definition~\ref{def:tropical_singularity}. It is straightforward to verify that
$\frac{\partial F}{\partial L_1}= x \frac{\partial F}{\partial x}=
a_{(1,0)}x+ a_{(1,1)}xy + 2 a_{(2,0)} x^2$,
$\frac{\partial F}{\partial L_2}= y \frac{\partial F}{\partial y}= a_{(0,1)}y+
a_{(1,1)}xy + 2 a_{(0,2)} y^2$, and
$\frac{\partial f}{\partial L_1}= 0 \odot w_1 \oplus 0 \odot w_1\odot w_2
\oplus 1 \odot w_1^2$, 
$\frac{\partial f}{\partial L_2}= 0 \odot w_2 \oplus 0 \odot w_1\odot w_2
\oplus 1 \odot w_2^2$, correspond to the standard partial derivatives.
Note that $q=(0,0) \in \mathcal{T}(f)$ is non singular but it also lies in the
intersection of the tropical hypersurfaces $\mathcal{T}(\frac{\partial
f}{\partial L_1})$ and $\mathcal{T}(\frac{\partial f}{\partial L_2})$.
Consider now the affine form $L_3 = w_1 -1$. Then, $\frac{\partial f}{\partial
L_3}= 1 \oplus 0 \odot w_2  \oplus 1 \odot w_1^2 \oplus 1 \odot w_2^2$ and $q =
(0,0)$ does not lie in the associated tropical hypersurface ${\mathcal
T}(\frac{\partial f}{\partial L_3})$ since the minimum of the linear forms
associated to the $4$ terms is attained only once at $q$.
\end{example}

The main result in this section is
Theorem~\ref{teo:singular_point_by_derivatives}, which characterizes singular
tropical hypersurfaces (with a given support) in terms of tropical Euler
derivatives. As we saw in Example~\ref{ex:nsconic}, it is not enough to consider
the $d$ Euler derivatives corresponding to the coordinate axes. It is not
difficult to solve this problem by appealing to the notion of a tropical basis
\cite{Computing_trop_var}, which we now recall.

\begin{definition}\label{def:tropical_basis}
Let $I\in \mathbb{K}[x_1^{\pm 1},\ldots, x_d^{\pm 1}]$ be an ideal. Then,
$Trop(I)$ consists of all those weights $w \in \R^d$ which satisfy the
following: $w \in \mathcal{T}(Trop(F))$ for every nonzero $F \in I$.
By~\cite{kapranov,stubook}, $Trop(I)$ coincides with $Trop(V(I))$, that is with
(the closure of) the image under the valuation map of the zeros $V_{\K^*}(I)$ of
$I$ in the torus $(\K^*)^d$. A tropical basis of $I$ is a finite set of
polynomials $F_1,\ldots, F_r$ generating $I$ such that $Trop(I) = \cap_{i=1}^r
\mathcal{T}(Trop(F_i))$.
\end{definition}

Given a finite lattice set $A\subseteq \mathbb{Z}^d$ with $n$ elements, we will
identify in what follows the space of polynomials with coefficients in
$\mathbb{K}$ and support $A$ with $(\mathbb{K}^*)^n$. Denote by $\overline{1}$
the point $(1,\ldots, 1)\in \mathbb{(K^*)}^d$. The subvariety \[H_{1}=\{F\in
(\mathbb{K}^*)^n| F \textrm{\ is\ singular\ at\ } \overline{1}\}\] of
polynomials with support $A$ and a singularity at $\overline{1}$ is a linear
space. Its closure in $\P^{n-1}(K)$ equals the dual space to the tangent space
at the point $\overline{1}$ of $X_A$. See the discussion of this space and the
following results in Section~\ref{sec:Bergman}. 

Denote by ${\mathcal L}$ the set of integer affine functions $L=j_1w_1+\ldots
+j_d w_d+\b$ such that  $\gcd(j_1,\ldots, j_d)=1$ and $ \dim \langle  \{L=0\}
\cap A\rangle = d-1$.

\begin{proposition}\label{prop:tbH1}

Let $(v_1, \dots, v_n)$ be variables.  The finite set of tropical linear
polynomials
\[P_1:=\left\{\bigoplus_{i\in A-\{L=0\}} 0 \odot v_i \, | \, L \in {\mathcal L} 
\right\}\]
is a tropical basis of $Trop(H_1)$.
\end{proposition}
\begin{proof}
Let $F=\sum_{i\in A} y_i x^i\in \mathbb{K}[x_1^{\pm 1},\ldots, x_d^{\pm 1}, y_i
(i\in A)]$ be the generic polynomial with support $A$. Note that  as $L$ runs
over all integer affine functions, the Euler derivatives  $\frac{\partial
F}{\partial L}$ are precisely {\em all} integral linear combination of $F, x_1
\frac{\partial F}{\partial x_1}, \dots, x_d \frac{\partial F}{\partial x_d}$.
$H_1$ is the linear space (in the variables $(y_1, \dots, y_n)$) defined by the
linear equations $F(x=\overline{1})$, $x_j \frac{\partial F}{\partial
x_j}(x=\overline{1})=0$, $1\leq j\leq d$. We know that the linear forms
vanishing on $H_1$  form a tropical basis of $H_1$ \cite{stubook}, and it is
enough to consider linear forms with rational (and a fortiori, integer) entries.
Now,  $Trop(\frac{\partial F}{\partial L}(x=\overline{1})) =  \bigoplus_{i\in
A-\{L=0\}} 0 \odot v_i$. Moreover, by \cite{Computing_trop_var}, the set of
linear forms in $H_1$ that have minimal support define a tropical basis of
$Trop(H_1)$. This set corresponds to the affine functions $L$ such that
$\{L=0\}\cap A$ spans an affine space of maximal dimension $d-1$.
\end{proof}

We have defined a tropical basis of the set of polynomials with a singularity at
$\overline{1}$. If we have another point $a\in (\mathbb{K}^*)^d$, we can easily
provide a tropical basis of the variety $H_a$ of hypersurfaces with a singular
point at $a$ by considering a diagonal change of coordinates. Explicitly, if
$F=\sum_{i\in A} a_i x^i$ is a Laurent polynomial with coefficients in $(\mathbb
K)^*$ with a singularity at $\overline{1}$ and such that $val(a_i)=p_i\odot
q_1^{i_1}\odot \ldots \odot q_d^{i_d}= p_i + \langle  i, q \rangle$, then the
polynomial $F_1=\sum_{i\in A} a_it^{-q_1i_1-\ldots -q_di_d} x^i$ has a
singularity at $(t^{q_1},\ldots,t^{q_d})$ and $val(a_it^{-q_1i_1-\ldots
-q_di_d})=p_i$.  We can easily deduce the following.

\begin{proposition}\label{prop:tbH}
Let $A\subseteq \mathbb{Z}^d$ with $\Z$-linear span $\mathbb{Z}^d$. As before,
identify $(\mathbb{K}^*)^n$ with the space of polynomials with support $A$.
Consider the incidence variety $H=\{(F,u)\in (\mathbb{K}^*)^n\times
(\mathbb{K}^*)^d| F {\rm\ is\ singular\ at\ }u\}$. Let $F=\sum_{i\in A} y_ix^i$
be the generic polynomial with support $A$, where $(x_1, \dots, x_d)$ and
$(y_i)_{i \in A}$ are variables. Then the finite set
\[P'=\{Trop(\frac{\partial F}{\partial L}) \text{ with } \langle \{L=0\}\cap
A\rangle \text{ of maximal dimension } d-1\},\]
is a tropical basis of $Trop(H)$.
\end{proposition}

We have now the tools to prove the following tropical characterization of
singular tropical hypersurfaces with fixed support.

\begin{theorem}\label{teo:singular_point_by_derivatives}
Let $f=\bigoplus_{i\in A} p_i \odot w^i$ be a tropical polynomial with support
$A$. Let $q\in \mathcal{T}(f)$ be a point in the hypersurface defined by $f$.
Then, $q$ is a singular point of $\mathcal{T}(f)$ if and only if $q\in
\mathcal{T}(\frac{\partial f}{\partial L})$ for all $L$. 

Thus, $f$ defines a singular tropical hypersurface if and only if \[\bigcap_{L}
\mathcal{T}(\frac{\partial f}{\partial L})\neq \emptyset.\] This intersection is
given by a finite number of Euler derivatives of $f$; for instance, we can take
only the affine linear functions $L \in {\mathcal L}$ defined before
Proposition~\ref{prop:tbH1}.
\end{theorem}

\begin{proof}
One implication is trivial. If $q$ is a singular point of $\mathcal{T}(f)$
there exists a polynomial $F=\sum_{i\in A} a_i x^i$, $val(a_i)=p_i$ with a
singularity at a point $b$ with $val(b)=q$. Then, $\frac{\partial F}{\partial
L}(b)=0$ for all $L$, and so $val(b)=q\in \mathcal{T}(\frac{\partial f}{\partial
L})$ for all $L$.

For the converse, let $q$ be a point in $\bigcap_{L} \mathcal{T}(\frac{\partial
f}{\partial L})$. In particular, $q\in \mathcal{T}(f)$. Then, 
for any integer affine function $L$, the minimum
$\min_{i\in A, L(i) \not=0}\{p_i + \langle q,i \rangle \}$
is attained at least twice. This happens if and only if for all $L$ 
the point $(p, q) \in \mathcal{T}( \bigoplus_{i \in A-\{L=0\}} v_i \odot w^i)$.
It follows from Proposition~\ref{prop:tbH} that $(p,q)$ belongs to the incidence
variety $Trop(H)$. So, by Kapranov's theorem~\cite{kapranov}, there is a point
$(F,b)\in V(H)$ such that $F$ is an algebraic polynomial with support in $A$ and
a singularity at $b$ such that $Trop(F)=f$ and $Trop(b)=q$.
\end{proof}

We present two examples that illustrate
Theorem~\ref{teo:singular_point_by_derivatives}.

\begin{example}\label{ex:first_example}
Let $A=\{(0,0),(1,0),(2,0),(1,1),(2,2),(0,2)\}$. Consider the tropical
polynomial $f=0\oplus 0\odot w_1\oplus 0\odot w_1^2 \oplus 0\odot w_1\odot w_2
\oplus 0\odot w_1^2\odot w_2^2 \oplus 6\odot w_2^2$. Let us compute its singular
points. Let $L_1=w_2-2$, $\frac{\partial f}{\partial L_1}=0 \oplus 0\odot w_1
\oplus 0\odot w_1^2 \oplus 0\odot w_1\odot w_2$, $L_2=w_2$, $\frac{\partial
f}{\partial L_2}=0\odot w_1\odot w_2 \oplus 0\odot w_1^2\odot w_2^2 \oplus
6\odot w_2^2$. The intersection set of these three curves is the segment $S$
whose ends are $(0,0), (3,-3)$. Consider now $L_3=w_1-w_2$ that contains all the
monomials \emph{dual} to $S$ (cf. the beginning of Section~\ref{sec:Bergman} for
a more precise explanation of this duality). Then, $\frac{\partial f}{\partial
L_3}= 0\odot w_1\oplus 0\odot w_1^2 \oplus 6\odot w_2^2$. The intersection of
this tropical curve with the segment $S$ is the set of points $\{(0,0),
(2,-2)\}$ (See Figure~\ref{fig:derivative1}). Let us check that these two points
are valid singular points. The polynomial
$F_1:=-1+4x+(-2+t^6)x^2+(-2-2t^6)xy+x^2y^2+t^6y^2$ has support $A$. It defines a
curve with a singularity at $(1,1)$ and $Trop(F_1)=f$. On the other side,
$F_2=(1-t^2+t^4)+(2-2t^2)x+x^2+(-2-2t^2)xy+x^2y^2+t^6y^2$ is a polynomial with
support $A$. $F_2$ defines a curve with a singularity at $(t^2,t^{-2})$ and such
that also $Trop(F_2)=f$. It is not difficult to see that it is not possible to
find a single polynomial $F$ with two singular points with valuations $(0,0)$
and $(2,-2)$.
\end{example}

\begin{figure}
\begin{center}
\includegraphics[width=0.4\textwidth]{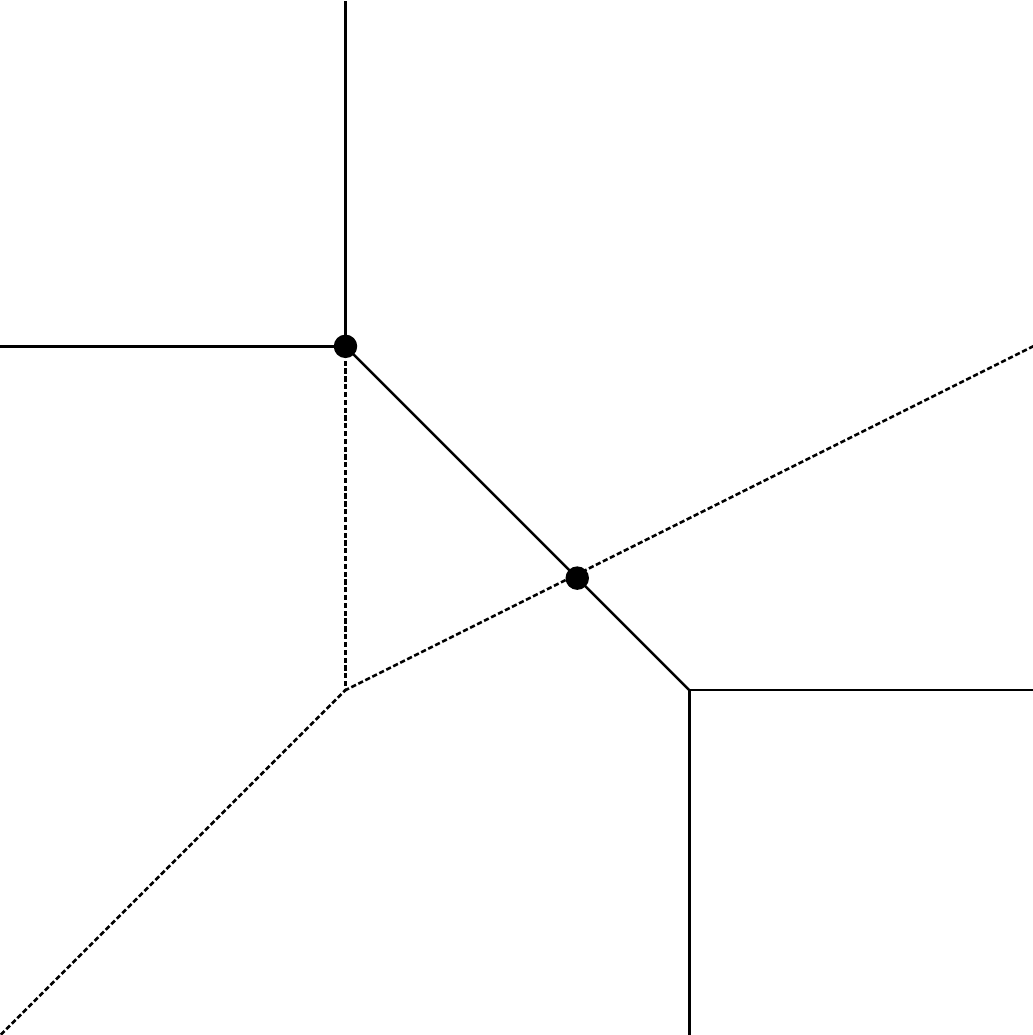}
\end{center}
\caption{The curve $f$ (bold) and $\frac{\partial f}{\partial (w_1+w_2)}$
(pointed) in Example~\ref{ex:first_example}}\label{fig:derivative1}
\end{figure}

\begin{example}\label{ex:laface}
Let $A=\{(0,0),(1,0),(2,0),(1,1),(2,2),(0,2)\}$ be as in
Example~\ref{ex:first_example} and let $A'=A \cup \{(2,10), (0,1), (1,2)\}$, so
that $A'$ are all the lattice points in the convex hull of $A$. Consider the
tropical polynomial $f'=0\oplus 0\odot w_1\oplus 0\odot w_1^2 \oplus 0\odot
w_1\odot w_2 \oplus 0\odot w_1^2\odot w_2^2 \oplus 6\odot w_2^2 \oplus 0\odot
w_1^2 w_2 \oplus 3\odot w_2 \oplus 3 \odot w_1 w_2^2$ with support in $A'$. That
is, the coefficients of the $3$ new points are given by interpolation of the
linear functions defining the subdivision associated to the polynomial $f$ in
the previous example. Note that all points in $A'$ are thus marked. It is easy
to check that all points in $\mathcal{T}(f')$ are singular. Indeed, in this case
$f'$ can be lifted to the polynomial $F = (1 + x + xy + t^3 y)^2$.
\end{example}

We now present Theorem~\ref{teo:singular_point_by_derivatives} into action in a
defective example, where the tropical $A$-discriminant can be explicitly
computed.

\begin{example}\label{ex:defective}
Let $A \subset \Z^3$ be the configuration $A = \{\alpha_1=(0,0,0),
\alpha_2=(1,0,0), \alpha_3 = (2,0,0), \alpha_4 = (0,0,1), \alpha_5 = (0,1,1),
\alpha_6 = (0,2,1)\}$. Thus, $A$ is the union of two one dimensional circuits
and the convex hull of $A$ is the lattice tetrahedron with vertices $\{\alpha_1,
\alpha_3, \alpha_4, \alpha_6\}$. Note that $A$ does not contain any circuit of
full dimension $3$. The zero set of any integer affine function $L$ such that
the affine span of $\{L=0\} \cap A$ is equal to $2$, consists of one of the
circuits plus one more point. Consider a tropical polynomial $f = \oplus_{\ell
=1}^6 p_{\alpha_i} \odot w^{\alpha_i}$ with support in $A$. Thus, there exists a
singular point $q \in \mathcal{T}(f)$ if and only if 
\begin{equation}\label{eq:pes}
2 p_{\alpha_2} = p_{\alpha_1} + p_{\alpha_3}, \quad 2 p_{\alpha_5} =
p_{\alpha_4} + p_{\alpha_6}.
\end{equation} 
This corresponds to the fact that this configuration is self-dual; indeed, the
dual variety $X_A^* \subset \P^5(\mathbb{K})$ has (projective) dimension $3$, it
is isomorphic to the toric variety $X_A$ and it is cut out by the binomials
$y_2^2 - 4 y_1 y_3 =0$, $y_5^2 = 4 y_4 y_6$, where $(y_1 : \dots : y_6)$ are
homogeneous coordinates in $\P^5(K)$. The tropicalization of this binomial ideal
is the rowspan of the associated matrix $A$ in $\R^6$, which is defined by the
equations~\eqref{eq:pes}.
\end{example}

\section{Marked tropical hypersurfaces and tropical
singularities}\label{sec:marked}

Given a tropical polynomial $f = \oplus_{i \in A} p_i \odot w^i$ with support
$A$, most of the (finite) Euler derivatives $\frac{\partial f}{\partial L}$ do
not provide relevant information to detect singular points of $\mathcal{T}(f)$.
In this section we give further conditions and characterizations to detect
singular points.

We need to recall the following duality \cite{GKZ-book}. The vector of
coefficients $p= (p_i)_{i \in A}$ of $f$ defines a \emph{coherent marked
subdivision} $\Pi_p$ of the convex hull $N(A)$ of $A$. That is, $p$ defines a
collection of subsets of $A$ (called \emph{marked cells}) which are in
one-to-one correspondence with the domains of linearity of the affine function
cutting the faces of the lower convex hull of the set of lifted points $\{(i,
p_i), i \in A\}$ in $\R^{d+1}$. Assume that a lower face $\Gamma_\varphi$ equals
the graph of an affine function $\varphi(w_1,\dots, w_n) = \langle  q_\varphi, w
\rangle + \beta_\varphi$. The corresponding marked cell $\sigma_\varphi$ of the
subdivision of $N(A)$ is the subset of $A$ of all those $i$ for which $p_i =
\varphi(i)$.

The marked subdivision $\Pi_p$ is combinatorially dual to the \emph{marked}
tropical variety $\mathcal{T}(f)$. As we saw, this is a polyhedral complex which
is a union of dual cells $\sigma_\varphi^*$, where we also record the
information of the dual cell $\sigma_\varphi$, and not only of the geometric
information of the vertices of $\sigma_\varphi$. More explicitly, the dual cell
$\sigma^*$ in $\mathcal{T}(f)$ of a given cell $\sigma$ of $\Pi_p$ equals the
closure of the union of the points $q_\varphi$ for all ways of writing $\sigma =
\sigma_\varphi$, and we also record the information of all the points in
$\sigma$, that is, of {\em all the points at which the minimum $f(q_\varphi)$ is
attained} for any point $q_\varphi$ in the relative interior of $\sigma^*$
(which is the \emph{marking} of the cell). The sum of the dimensions of a pair
of dual cells is $d$. In particular, vertices of $\mathcal{T}(f)$ correspond to
marked cells of $\Pi_p$ of maximal dimension $d$.

We now prove that $\Pi_p$ is a (coherent) triangulation, then the tropical
hypersurface associated to $f$ is non singular, as expected. As we will see, the
converse to this statement is not true and involves a complicated combinatorial
study. Recall that a point configuration is a pyramid, if all but one of its
points lie in an affine hyperplane.

\begin{lemma}\label{lem:punto_regular}
Let $q \in {\mathcal T}(f)$ lying in the relative interior of a cell $\sigma^*$
such that the dual cell $\sigma$ in $\Pi_p$ is a pyramid. Then, $q$ is non
singular. In particular, if $\Pi_p$ is a coherent triangulation, then the
tropical hypersurface ${\mathcal T}(f)$ is non singular.
\end{lemma}

\begin{proof}
If $\sigma$ is a pyramid, let $L$ be a linear functional such that $\{L=0\}$
intersects $\sigma$ in a facet and leaves out one point. This means that the
minimum of $\frac{\partial f}{\partial L}$ at $q$ is attained at exactly one
monomial. Hence, $q\notin \mathcal{T}(\frac{\partial f}{\partial L})$ and $q$ is
not a singular point.
\end{proof}

\begin{example}[Example~\ref{ex:nsconic}, continued]\label{ex:triang2}
Consider the same configuration $A$ of Example~\ref{ex:nsconic}, that is, the
six lattice points of the $2$-simplex in the plane and the tropical conic $g = 0
\oplus 1 \odot w_1 \oplus 1 \odot w_2 \oplus 1 \odot w_1\odot w_2 \oplus 0 \odot
w_1^2 \oplus 0 \odot w_2^2$. Then the associated marked subdivision has only one
cell $\sigma= \{ (0,0), (2,0), (0,2)\}$ and it is not singular by
Lemma~\ref{lem:punto_regular}, even if there are points in $A$ that do not occur
in the subdivision and this single cell has lattice volume bigger than one.
Also, it is straightforward to check that any polynomial $G =a_{(0,0)}+
a_{(1,0)} x + a_{(0,1)} y +$ $a_{(1,1)} xy + a_{(2,0)} x^2+ a_{(0,2)} y^2 \in
\K[x,y]$ with the given valuations is non singular, because this time the only
term in the expansion of the determinant of the matrix \eqref{eq:matrix} with
smallest valuation $0$, is the diagonal term $8 a_{(2,0)} a_{(0,2)} a_{(0,0)}$.
\end{example}

\begin{example}\label{ex:nontriang2} 
Consider the tropical polynomial $f=0\oplus 0 \odot w_1\oplus 0\odot w_2 \oplus
1 \odot w_3 \oplus 0\odot w_1\odot w_3 \oplus 0\odot w_2 \odot w_3$. We read the
support $A$ from $f$. The coefficients of $f$ induce the coherent marked
subdivision depicted in Figure~\ref{fig:cayley2rectas}, which has two cells of
dimension $3$ and it is not a triangulation.
\begin{figure}
\begin{center}
\includegraphics[width=0.4\textwidth]{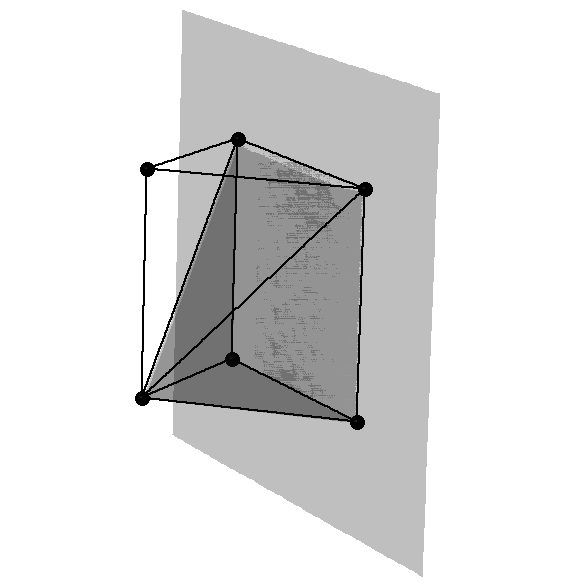}
\end{center}
\caption{The subdivision of
Example~\ref{ex:nontriang2}}\label{fig:cayley2rectas}
\end{figure}
One of these top dimensional cells is a unimodular 3-simplex. The second top
dimensional cell contains a circuit $Z= \{(1,0,0), (0,1,0), (1,0,1), (0,1,1)\}$
of dimension $d-1 =2$ and it is a pyramid over the point $(0,0,0)$. The affine
integer function $L = 1- w_1 -w_2$ verifies that $Z \subset \{L=0\}$. Computing
$\frac{\partial f}{\partial L}$, we check that $f$ defines a \emph{non singular}
tropical surface.
\end{example}

We now analyze some further conditions that a point $q\in \mathcal{T}(f)$ must
satisfy in order to be a tropical singular point.

\begin{theorem}\label{teo:char_punto_singular}
Let $f=\bigoplus_{i\in A} p_i \odot w^i$ be a tropical polynomial and $q\in
\mathcal{T}(f)$ lying in the interior of  a cell $\sigma^*$. Then, $q$ is a
singular point if and only if the dual cell $\sigma$ is not a pyramid and we
have that $q\in \mathcal{T}(\frac{\partial f}{\partial L})$ for all affine
linear functions $L$ such that $\dim \langle  \{L=0\}\cap {A}\rangle = d-1$ and
$\sigma \subseteq \{L=0\}$. So, in the particular case of a vertex $q$ of
$\mathcal{T}(f)$,
$q$ is singular if and only if $\sigma$ is not a pyramid.
\end{theorem}

\begin{proof}
If $q$ is a singular point, $\sigma$ is not a pyramid by
Lemma~\ref{lem:punto_regular}. As $q\in \mathcal{T}(\frac{\partial f}{\partial
L})$ for any $L$ by definition, in particular this happens for those $L$ of the
form described in the hypotheses. Suppose now that $q$ is not a singular point
and let $L'$ be an affine integer function such that $q\notin
\mathcal{T}(\frac{\partial f}{\partial L'})$. Let $i \in A - \{L'=0\}$ be the
unique point of $A$ at which $\mathcal{T}(\frac{\partial f}{\partial L'})(q)$ is
attained. Then, if $\sigma$ is not contained in  $\{L'=0\}$,  we have that $i
\in \sigma$ and it is the unique point of $\sigma$ outside $\{L'=0\}$, and so
$\sigma$ is a pyramid. Otherwise, take any integer hyperplane $\{L=0\}$ such
that $A\cap \{L'=0\}\subseteq A\cap \{L=0\}$, $A\cap \{L=0\}$ spans an affine
space of dimension $d-1$ and $i\notin L$. For any such $L$ we have that $q\notin
\mathcal{T}(\frac{\partial f}{\partial L})$, as wanted.
\end{proof}

As a consequence, we can easily describe the polynomials that define singular
hypersurfaces in the case of $1$ and $2$ variables. Recall that, if $A$ is not
defective, then ${\mathcal T}(Trop(\Delta(A)))$ is a subfan of the secondary fan
of $A$. In the simplest case of one variable, $A\subseteq \mathbb{Z}$, it holds
that ${\mathcal T} (Trop(\Delta_A))$ equals the union of the non top dimensional
cones in the secondary fan (since the only proper faces of $A$ are vertices). 
Hence, a univariate polynomial is singular if and only if the induced marked
subdivision is not a triangulation. With our notation, this is a simple case of
Theorem~\ref{teo:char_punto_singular}, because all circuits of $A$ are of
maximal dimension $1$. 

The following result, in the smooth case, appears in \cite[Prop. 3.9, Ch.
11]{GKZ-book}. 

\begin{corollary}\label{teo:singular_dim_1_and_2}
Let $A\subseteq \mathbb{Z}^2$ with $n$ elements. Suppose $p\in \R^n$ induces a
coherent marked subdivision $\Pi_p$ in $A$ that is not a triangulation. Then $p$
is in ${\mathcal T}(Trop(\Delta_A))$ (equivalently, the polynomial $f =
\oplus_{i\in A} p_i \odot w^i$ defines a singular tropical hypersurface) in
exactly the following situations:
\begin{itemize}
\item[i)] 
There exists a marked cell of $\Pi_p$ which contains a circuit of dimension $2$.
\item[ii)] \label{case:ii} All circuits contained in a cell of $\Pi_p$ have
affine dimension $1$ and there exists a marked cell $\sigma$ of $\Pi_p$ of
dimension $1$ and cardinality $|\sigma| \ge 3$ with the following property: Let
$L$ be an integer affine function such that $\sigma \subset\{L=0\}$. Then,
$\sigma^* \cap \mathcal{T}(\frac{\partial f}{\partial L}) \not= \emptyset$.
\end{itemize}
\end{corollary}

The first item is contained in Theorem~\ref{teo:char_punto_singular}. With
respect to the second item, note that $\sigma$ contains a circuit $Z$ (of
dimension $1$) and for any integral affine function $L$, $Z \subset \{L=0\}$ if
and only if $\sigma \subset \{L=0\}$. The result follows again by
Theorem~\ref{teo:char_punto_singular}.

In case item ii) of Proposition~\ref{teo:singular_dim_1_and_2} holds, $\sigma^*
\cap \mathcal{T}(\frac{\partial f}{\partial L}) \not= \emptyset$ if and only if
there is a cell $\sigma'$ of dimension $2$ containing $\sigma$, such that
$\sigma' \cap \{L \not= 0\} = \{ i_1\}$ consists of a single point $i_1 \in A$
and, assuming $L(i_1) > 0$,  there exists another point $i_2 \in A - \sigma'$
with $L(i_2) < L(i_1)$.  This is a particular case of the following more general
result. Recall that we always assume that the convex hull of our exponent set
$A$ is full dimensional.

\begin{proposition} \label{prop:anotherpt}
Let $A\subseteq \mathbb{Z}^d$. Let $p \in \R^n$ such that $\Pi_p$ contains a top
dimensional cell $\sigma'$ which contains a circuit $Z$ of dimension $d-1$ and
it is a pyramid over a point $i_1$. Let $L$ be an affine integer function such
that $Z \subset \sigma'\cap \{L=0\}$ and $L(i_1)> 0$.  Then, there exists a
singular point $q \in {\mathcal T} (\oplus_{i\in A} \, p_i \odot w^i) \cap
\{(\sigma'\cap \{L=0\})^*\}$  with $\langle  q, i_1 \rangle >0$ if and only if
there exists another point $i_2 \in A$ not in $\sigma'$ such that $L(i_2) <
L(i_1)$. In particular, if $Z$ intersects the interior of $N(A)$, then
${\mathcal T}(\oplus_{i\in A} \, p_i \odot w^i)$ is singular.
\end{proposition}

\begin{proof}
We can assume that $ L(w) = j_1 w_1 + \dots + j_d w_d + \b$, with $j_1, \dots,
j_d$ coprime. To make the notation easier, we apply an invertible affine linear
trans\-formation to our configuration $A$ so that $L (w) = w_1$.  Denote by
$\varphi(w) = \varphi_1 w_1 + \dots +\varphi_d w_d$ the linear form which
interpolates $p$ over the cell $\sigma'$, that is $\varphi(i) = p_i$ for all $i
\in \sigma'$ and $\varphi(i) < p_i$ for all $i \notin \sigma'$. Thus,  $p'\in
\R^n$ defined by $p'_i:= p_i - \varphi(i)$ defines the same marked subdivision.
So we can assume that $p_i =0$ for all $i \in \sigma'$ and $p_i > 0$ otherwise.
Therefore, $q=(0,\dots,0)$ is the vertex of $\mathcal T$ dual to $\sigma'$,
which is not singular since it does not lie in ${\mathcal T}(\frac{\partial
f}{\partial L})$. There will be a singular point $q= (q_1, 0, \dots, 0)$ in
$(\sigma'\cap \{L=0\})^*$ with $\langle  q, i_1 \rangle > 0$ if and only if
there exists $q_1 > 0$ and two points $i_2,i_3$ in $A$ such that 
\begin{equation*}
 q_1 L(i_2) + p_{i_2} =  q_1 L(i_3) + p_{i_3}  \leq q_1  L(i) + p_i,
\end{equation*}
for all $i \in A$ with $L(i) \not =0$. Note that as $\sigma'$ is a pyramid over
$i_1$, for any point $i_2$ in $A$ for which $L(i_1)=L(i_2)$ it holds that $i_2$
is not in $\sigma'$, or equivalently, $p_{i_2} >0$. Assume first that there is a
point $i$ in $A':=A - \sigma'$ with $L(i) < L(i_1)$ and let $i_2$ with these
properties and such that moreover $\frac {p_{i_2}}{L(i_1) - L(i_2)} = {\rm
min}_{ i \in A'}  \frac {p_{i}}{L(i_1) - L(i)}$. Then, it is enough to take $q_1
= \frac {p_{i_2}}{L(i_1) - L(i_2)}$ and $i_3 = i_1$. Reciprocally, assume there
exists a singular point $q = (q_1, 0, \dots, 0)$ with $q_1 > 0$. As $q \in
{\mathcal T}(\frac{\partial f}{\partial L})$, there exist two points $i_2 \neq
i_3$ such that $q_1 L(i_2) +p_{i_2} = q_1 L(i_3) + p_{i_3} \le q_1 L({i_1})$.
Assume $i_2 \not=i_1$. Then, $0 < p_{i_2} \le q_1 (L({i_1}) - L({i_2}))$.
Therefore, $L({i_2}) < L({i_1})$, as wanted. The condition that $Z$ intersects
the interior of $N(A)$ guarantees the existence of a point $i_2 \in A - \sigma'$
with $L({i_2}) < L({i_1})$.
\end{proof}

Note that the point $i_2$ in the statement of Proposition~\ref{prop:anotherpt}
does not need to belong to any cell in $\Pi_p$.

\section{Weight classes and the co-Bergman fan of $\widetilde{A}$}
\label{sec:Bergman}

In this section, we relate our definitions to the results and definitions
in~\cite{Bergman-complex,tropical-discriminant,sing-fixed-point}. As before, $f
=\bigoplus_{i\in A} p_i \odot w^{\odot i} \in \mathbb{R}[w_1, \ldots, w_d]$
denotes a tropical polynomial with  support $A$.

\begin{definition}\label{def:weight_class}
Let $q$ be in the interior of a cell $\sigma^* \subseteq \mathcal{T}(f)$. We
define the \emph{flag} of $f$ with respect to $q$ as the flag of subsets
${\mathcal F}(q)$ of $A$ defined inductively by: $F_0(q)=\sigma \subsetneq
F_1(q)\subsetneq \ldots \subsetneq F_r(q)$, $dim \langle  F_r(q) \rangle =d$,
and for any $\ell$:
$F_{\ell+1}(q) - F_\ell(q)$ is the subset of $A- \langle F_\ell(q)\rangle$ where
the tropical polynomial $\bigoplus_{i\in A - \langle F_\ell(q)\rangle }\, p_i
\odot w^i$ attains its minimum at $q$. The {\em weight class} of the flag
$\mathcal{F}(q)$ are all the points $q'\in \mathcal{T}(f)$ for which
$\mathcal{F}(q) =\mathcal{F}(q')$ 
\end{definition}

Theorems~\ref{teo:singular_point_by_derivatives}
and~\ref{teo:char_punto_singular} provide an algorithm to decide if $q\in
\mathcal{T}(f)$ is singular or not, which is similar to the method presented in
\cite{Tesis-master-Ochse} but which works without any restrictive hypothesis on
$A$. The algorithm returns an $L$ such that $q\notin \mathcal{T}(\frac{\partial
f}{\partial L})$ or ``$q$ is a singular point". First, we compute $F_0(q) =
\sigma$. If $\sigma$ is a pyramid,  there exists $i \in F_0(q)$ such that
$i\notin \langle F_0(q) - \{i\}\rangle$ and we can compute an $L$ defining the
facet $\langle F_0(q) -\{i\}\rangle$ of $F_0(q)$, which verifies $q\notin
\mathcal{T}(\frac{\partial f}{\partial L})$. If this is not the case and the
dimension of $\langle  F_0(q) \rangle < d$, we compute $F_1(q)$ and we iterate
the procedure.  We stop when we find  an $L$ that certifies that $q$ is not
singular or when $F_\ell$ spans an affine dimension $d$, in which case $q$ is
singular.

Number the elements $i_1 \dots, i_n$ of $A$ and call $A \in \Z^{d\times n}$
matrix with these columns. Let $\widetilde{A}$, as in the Introduction, be the
integer matrix with columns $(1,i_k), k=1, \dots, n$. Thus, the vector
$\overline{1}=(1, \dots, 1)$ lies in the row span of $\widetilde{A}$. In fact,
as the $A$-discriminant is an affine invariant of the configuration $A$, we
could assume without loss of generality that $A$ has this property, but we
prefer to point out the fact the we are interested in affine properties of the
configuration $A$, equivalent to linear properties of $\widetilde{A}$.

Let $L(w) =j_1 w_1+\ldots+ j_d w_d +\b$ be an affine linear function. We can
associate to $L$ the linear form $\ell_L(x_1,\dots, x_n) = \sum_{k=1}^n L(i_k)\,
 x_k $. Then, the support of $\ell_L$ is precisely $A - \{ L=0\}$. Moreover, the
coefficient vector $(L(i_k))_{k=1}^n$ lies in the row span of $\widetilde{A}$,
as it is obtained as the product $(\b, j_1, \dots, j_d) \cdot \widetilde{A}$,
and all linear forms in the row span  of $\widetilde{A}$ (which span the ideal
of $\ker(\widetilde{A})$), are of this form. Let ${\mathcal B} \subset \Z^c$ be
a Gale dual configuration of $\widetilde{A}$. The linear forms in the row span
of $\widetilde{A}$ with minimal support correspond to the circuits in $\mathcal
B$ and to the affine linear forms $L$ such that $\dim \langle A \cap
\{L=0\}\rangle = d-1$. 

Denote by $v_1, \dots, v_n$ (tropical) variables. The tropicalization
$Trop(\ell_L)$ equals:
\[Trop(\ell_L)(v) \, = \, \bigoplus_{L(i_k) \not =0} 0 \odot v_k. \]
We recover the fact that $H_1$ is an incarnation of $\ker(\widetilde{A})$ and so
$Trop(H_1)$ equals the co-Bergman fan ${\mathcal B}^*(\widetilde{A})$ (cf.
Proposition~\ref{prop:tbH1}). The flag of sets $\mathcal F(q)$ and the weight
classes in Definition~\ref{def:weight_class} coincide for instance with those
occurring in~\cite[Page~3]{Bergman-complex}.

Our previous algorithm can be modified to decide whether $f = \bigoplus_{i\in
A}\, p_i \odot w^i$ contains a singular point, that is, to decide whether $p$
lies in the tropical $A$-discriminant, and in this case, to compute all the
singular points. Just notice that, as weight classes induce a fine subdivision
on the co-Bergman fan {$\mathcal{B}^*(\widehat{A})$}, they also induce a finer
polyhedral subdivision of $\mathcal{T}(f)$. Two points $q$ and $q'\in
\mathcal{T}(f)$ belong to the relative interior of the same cell of the fine
subdivision if and only if $q,q'$ belong to the same weight class. If $\sigma$
is a cell of the fine subdivision of $\mathcal{T}(f)$, then either every point
of $\sigma$ is singular or all points are regular. Since the number of cells in
this subdivision is finite and computable, we can derive an algorithm to compute
all singular points of $\mathcal{T}(f)$ that uses this information.

\begin{proposition}\label{prop:singcells}
The (finitely many) weight classes associated to a tropical polynomial $f =
\bigoplus_{i\in A}\, p_i \odot w^i$ with support $A$, are relatively open
polyhedral cells which refine the polyhedral structure of $\mathcal{T}(f)$ dual
to the marked coherent subdivision $\Pi_p$. If $C$ is a cell in this new
subdivision, then all points in $C$ are singular or all of them are regular. The
previous algorithm applied to any of the points in $C$, allows us to decide if
$C$ is a set of singular or regular points.
\end{proposition}

We can thus reprove~\cite[Theorem~1.1]{tropical-discriminant}: a point $p
=(p_i)_{i\in A}$ lies in the tropicalization $Trop(X_A^*)$ of the
$A$-discriminant if and only if there exists a singular point $q \in
\mathcal{T}(\bigoplus_{i\in A} p_i \odot w^i)$.  This happens if and only if the
$n$-th dimensional vector $v =(v_i)_{i\in A}= p + \langle q, \cdot \rangle$
defined by the equalities
\[ v_i = p_i + \langle q, i \rangle, \quad i \in A,\]
lies in $\mathcal{B}^*(\widehat{A})$, by the characterization given in
Theorem~\ref{teo:singular_point_by_derivatives} (expressed by the previous
algorithm). Equivalently, if and only if 
\begin{equation} \label{eq:pvq}
p = v + \langle -q, \cdot \rangle 
\end{equation}
with $v$ in  $\mathcal{B}^*(\widehat{A})$ and $q \in \R^d$. That is, if and only
if $p$ lies in the Minkowski sum of $\mathcal{B}^*(\widehat{A})$ and the row
span of $A$, which equals the Minkowski sum of $\mathcal{B}^*(\widehat{A})$ and
the row span of $\widetilde{A}$, since the vector $\overline{1} \in
\mathcal{B}^*(\widehat{A})$. It follows that if $p \in \mathcal{T}(X_A^*)$, the
singular points of $\mathcal{T}(\bigoplus_{i\in A} p_i \odot w^i)$ are those $q$
which occur in a decomposition of the form~\eqref{eq:pvq}.

\medskip

We end this section with some examples that exhibit different interesting
features of these objects. The next example shows a tropical polynomial $f$
whose coefficients lie in a codimension one cone of the secondary fan of $A$,
for which $\mathcal{T}(f)$ has two singular points. 

\begin{example}\label{ex:example_discriminant}
Let $f=0\odot w_1^2\oplus 0\odot w_1^2\odot w_2\oplus 0\odot w_1^2\odot w_2^2
\oplus 7\odot w_2 \oplus 4\odot w_1\odot w_2$ $\oplus 7\odot w_1^4\odot w_2$.
The subdivision induced by $f$ in its Newton Polygon is a triangulation except
for the circuit of exponents in $\{w_1^2, w_1^2w_2, w_1^2w_2^2\}$. The
$A$-discriminant of the support of $f$ is (with the obvious meaning of the
variables $a_{ij}$)

{\tiny
\noindent
$\underline{256 a_{01}^{2} a_{21}^{8}}-192 a_{11}^{4} a_{20} a_{21}^{4} a_{22}
$  $-\underline{4096 a_{01}^{2} a_{20} a_{21}^{6} a_{22}} -$ 
$6144 a_{01} a_{11}^{2} a_{20}^{2} a_{21}^{3} a_{22}^{2} +$
$\underline{24576 a_{01}^{2} a_{20}^{2}a_{21}^{4} a_{22}^{2}}$ $-
$ $1024 a_{11}^{4} a_{20}^{3}$ $a_{22}^{3} +$ $8192 a_{01} a_{11}^{2} a_{20}^{3}
a_{21} a_{22}^{3} - $ $\underline{65536a_{01}^{2} a_{20}^{3} a_{21}^{2}
a_{22}^{3}} +$ $\underline{65536a_{01}^{2} a_{20}^{4} a_{22}^{4}}
+$ $216 a_{11}^{6} a_{21}^{3} a_{41} - $ $2016 a_{01} a_{11}^{4} a_{21}^{4}
a_{41}$ $+$ $5632$ $a_{01}^{2} a_{11}^{2} a_{21}^{5} a_{41} -$ $4096 a_{01}^{3}$
$a_{21}^{6}$ $a_{41} +$ $2592$ $a_{11}^{6}$ $a_{20} a_{21} a_{22} a_{41} - $
$20736 a_{01} a_{11}^{4} a_{20} a_{21}^{2} a_{22} a_{41} +$ $28672 a_{01}^{2}
a_{11}^{2} a_{20}$ $a_{21}^{3}$ $a_{22} a_{41} +$ $16384 a_{01}^{3}$ $a_{20}
a_{21}^{4} a_{22} a_{41} +$ $4608 a_{01} a_{11}^{4} a_{20}^{2} a_{22}^{2} a_{41}
- $ $204800 a_{01}^{2} a_{11}^{2} a_{20}^{2}$ $a_{21}$ $a_{22}^{2} a_{41} +$
$65536 a_{01}^{3}$ $a_{20}^{2}$ $a_{21}^{2} a_{22}^{2} a_{41} - $ $262144
a_{01}^{3} a_{20}^{3} a_{22}^{3} a_{41} +$ $729 a_{11}^{8} a_{41}^{2} - $ $7776$
$a_{01}$ $a_{11}^{6}a_{21} a_{41}^{2} +$ $27648 a_{01}^{2}$ $a_{11}^{4}$
$a_{21}^{2}$ $a_{41}^{2} -$ $38912 a_{01}^{3} a_{11}^{2}$ $a_{21}^{3} a_{41}^{2}
+24576 a_{01}^{4} a_{21}^{4} a_{41}^{2} -55296$ $a_{01}^{2} a_{11}^{4} a_{20}
a_{22} a_{41}^{2} +$ $122880$ $a_{01}^{3}$ $a_{11}^{2} a_{20} a_{21} a_{22}
a_{41}^{2} +$ $65536$ $a_{01}^{4} a_{20}a_{21}^{2} a_{22} a_{41}^{2}$ $+393216
a_{01}^{4} a_{20}^{2} a_{22}^{2}a_{41}^{2} - $ $13824 a_{01}^{3}$ $a_{11}^{4}
a_{41}^{3} +$ $73728 a_{01}^{4}a_{11}^{2} a_{21} a_{41}^{3}$ $-65536 a_{01}^{5}$
$a_{21}^{2}$ $a_{41}^{3}$ $-262144a_{01}^{5} a_{20} a_{22} a_{41}^{3} +$ $65536$
$a_{01}^{6}$ $a_{41}^{4}+ $ $768 a_{11}^{4} a_{20}^{2} a_{21}^{2}
a_{22}^{2}+$ $16 a_{11}^{4} a_{21}^{6} -$ $128 a_{01} a_{11}^{2} a_{21}^{7}= $
$1536 a_{01} a_{11}^{2} a_{20} a_{21}^{5} a_{22} $.}

The minimum valuation of the terms in the $A$-discriminant is attained for any
choice of coefficients $a_{ij}$ with valuations prescribed by the coefficients
of $f$, in the five underlined monomials of the $A$-discriminant. Three of these
monomials $a_{01}^{2} a_{20}^{2} a_{21}^{4} a_{22}^{2}, a_{01}^{2} a_{20}
a_{21}^{6} a_{22}, a_{01}^{2} a_{20}^{3} a_{21}^{2} a_{22}^{3}$ lie in the
convex hull of the other two $a_{01}^{2} a_{21}^{8}, a_{01}^{2} a_{20}^{4}
a_{22}^{4}$. Hence, the exponents of the monomials of the $A$-discriminant where
the minimum is attained lie on an edge, and the vector of coefficients of $f$
belongs to a maximal cell of the tropicalization of the $A$-discriminantal
variety. The singular points of this curve are $(3,0),(-1,0)$ (See
Figure~\ref{fig:example_discriminant}). Two lifts of the curve and the singular
point are: $t^{7} x^{4}y + x^{2}y^{2} + (-3 t^{13} + t - 2) x^{2}y + x^{2} + (2
t^{16} - 2 t^{4}) xy + t^{7} y$ with a singularity at $(t^3,1)$, and $t^{7}
x^{4}y + x^{2}y^{2} + (t^{9} - 3 t^{5} - 2) x^{2}y + x^{2} + (-2 t^{8} + 2
t^{4}) xy + t^{7} y$ with singularity at $(1/t,1)$.
\end{example}
\begin{figure}
\begin{center}
\includegraphics[width=0.4\textwidth]{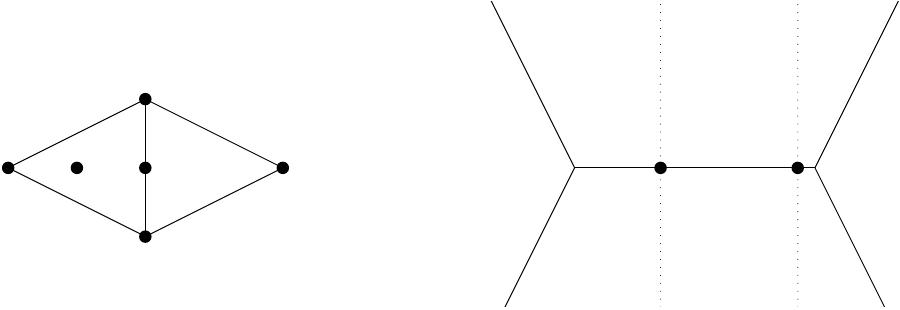}
\end{center}
\caption{Newton subdivision and singular curve of
Example~\ref{ex:example_discriminant} (in bold) and $\frac{\partial f}{\partial
(w_1-2)}$ (pointed)}\label{fig:example_discriminant}
\end{figure}

Our next example shows that for two vectors of coefficients inducing the same
coherent subdivision of $A$, the associated flags need not coincide.

\begin{example} \label{ex:howDareformed}
Let $A =\{\alpha_1=(0,0),$ $\alpha_2=(0,1),$ $\alpha_3 = (0,2),$ $\alpha_4 =
(2,0),$ $\alpha_5 = (1,2), \alpha_6 = (-2,0)\}$ and $p_v = (0,0,0,0,v_1, v_2)$,
with $v=(v_1, v_2) \in R_{> 0}^2$ arbitrary. In this case, $p_v$ defines the
curve given by $f_v=0 \oplus 0\odot w_2 \oplus 0\odot w_2^2 \oplus 0\odot w_1^2
\oplus v_1\odot w_1\odot w_2^2 \oplus v_2\odot w_1^{-2}$. The marked subdivision
$\Pi$ induced by any $p_v$ contains three maximal cells: $\sigma_1 = \{\alpha_1,
\alpha_2, \alpha_3, \alpha_4\}$, $\sigma_2=\{\alpha_3, \alpha_4, \alpha_5\}$,
$\sigma_3=\{\alpha_1, \alpha_2, \alpha_3, \alpha_6\}$. We claim that all these
curves are singular, with a singular point in the cell dual to the marked edge
$\{\alpha_1, \alpha_2, \alpha_3\}$. However, as we will see, the number and
locations of the singular points vary.

Cell $\sigma_1$ is dual to the point $(0,0)$ in the curve and $\sigma_3$ is dual
to the point $(v_1/2, 0)$. By Theorem~\ref{teo:singular_dim_1_and_2} there is a
singularity if there is a point in the segment $[(0,0), (v_2/2,0)]$ that also
belongs to the partial derivative $g_v=\frac{\partial f_v}{\partial (w_1=0)} =
0\odot w_1^2 \oplus v_1\odot w_1\odot w_2^2 \oplus v_2\odot w_1^{-2}$. In the
segment, $g_v$ attains its minimum at $(0,0)$ on the linear form associated to
$\alpha_4$ and $g_v$ attains its minimum at $(v_2/2, 0)$ on the linear form
associated to $\alpha_6$. Since $g_v$ is a continuous function, there must be a
point $(q, 0)$ where the minimum of $g_v$ is attained twice, so this point will
be a singularity of $f$ (cf. \cite{sing-fixed-point}). This reasoning works for
any hypersurface in dimension $d$ with a circuit in the interior of $A$ of
dimension $d-1$.
\begin{itemize}
\item If $-4v_1 + v_2 < 0$ there is a singular point at $q=(v_2/4, 0)$, the
flag with respect to $q$ is: $\{\alpha_1,\alpha_2,$
$\alpha_3\}$
$\subsetneq$ $\{\alpha_1, \alpha_2, \alpha_3, \alpha_4,\alpha_6\} \subsetneq A$.
\item If $-4v_1 + v_2 = 0$ there is a singular point at $q=(v_2/4, 0)$, the
flag with respect to $q$ is: $\{\alpha_1,
\alpha_2,\alpha_3\} \subsetneq A$
\item If $-4 v_1 + v_2 > 0$ we get two different singular points:
\begin{itemize}
\item $q=(v_1, 0)$ with flag with respect to $q$:\\
$\{\alpha_1,\alpha_2,\alpha_3\} \subsetneq \{\alpha_1, \alpha_2, \alpha_3,
\alpha_4, \alpha_5\} \subsetneq A$.
\item $q=((v_2-v_1)/3, 0)$ with flag with respect to
$q$:\\
$\{\alpha_1,\alpha_2,\alpha_3\} \subsetneq \{\alpha_1, \alpha_2, \alpha_3,
\alpha_5, \alpha_6\} \subsetneq A$.
\end{itemize}
\end{itemize}
\begin{figure}
\begin{center}
\includegraphics[width=0.9\textwidth]{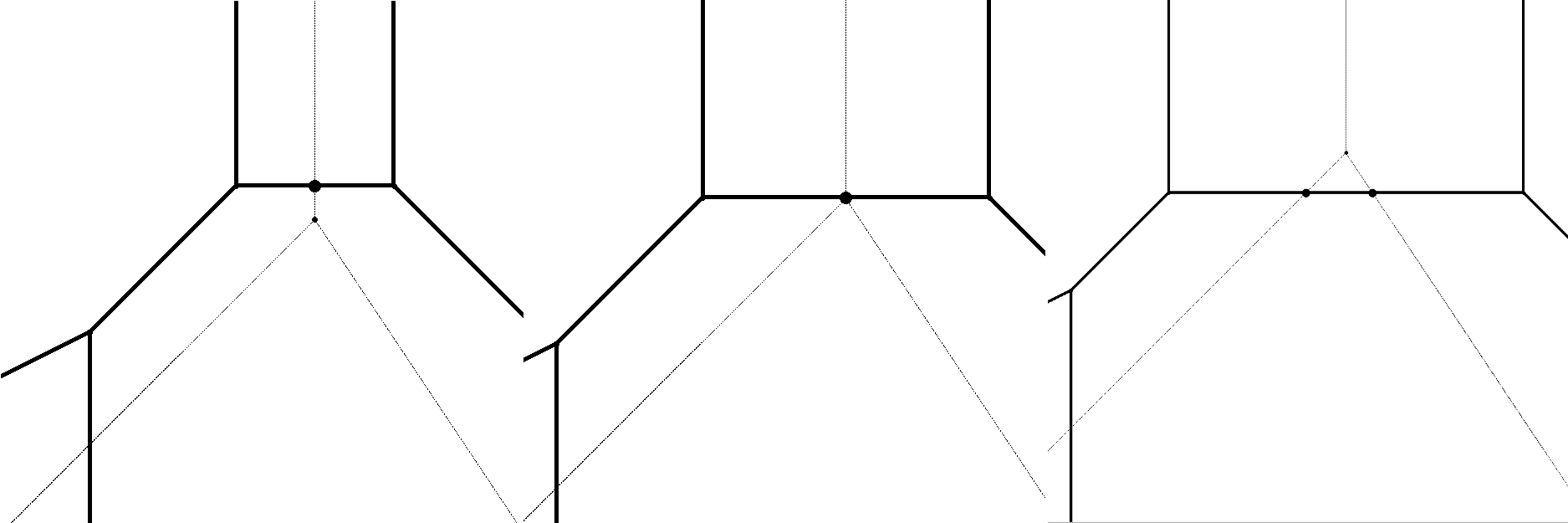}
 $-4v_1+v_2<0$ (left), $-4v_1+v_2=0$ (center), $-4v_1+v_2 > 0$ (right)
\end{center}
\caption{Cases of Example~\ref{ex:howDareformed}}\label{fig:howDareformed}
\end{figure}
Thus, we can take different values of $v, v' \in \R_{>0}^2$ such that, while
keeping $\Pi= \Pi_{p_v} = \Pi_{p_{v'}}$ invariant, it is not possible to find
singular points $q_v, q_{v'}$ in $\mathcal{T}(f_v), \mathcal{T}(f_{v'})$ for
which the flags coincide $F_\ell(q_v)=
F_\ell(q_{v'})$ for all $\ell$. 
\end{example}

It is worthwhile to note that the set of singular points in a hypersurface is
not, in general, a tropical variety. 

\begin{example}
Let $f=0\oplus 0\odot w_2 \oplus 0 \odot w_2^2 \oplus 0\odot w_1 \oplus 0 \odot
w_1\odot w_2 \oplus 1 \odot w_1^2$ represents a tropical conic. So, if it
is singular, it is a pair of lines. It happens that this conic is the union of
the lines $0\oplus 0\odot w_1\oplus 0\odot w_2$ and $0\oplus 1\odot w_1\oplus
0\odot w_2$. The intersection of these two lines is the ray $(0,0)+ p (1,0)$, $p
\geq 0$. This is not a tropical variety. However, any of its points is a valid
singular point of the conic. We can take any point $q$ in the intersection set
and lift the whole configuration, as it is an acyclic configuration (see
\cite{Transfer-trop}). In fact, for the point $(p,0)$, $p\geq 0$ we can take the
lift $F=(1 +x -(1+t^p)y) (1 +tx -(1 +t^{p+1})y) =1 +(1 +t)x +(-2 -t^p-t^{p+1})
y+ tx^2+( -1 -t -t^{ p+1}) xy+( 1+ t^p +t^{p+1} +t^{2p+1}) y^2$ that has a
singularity at $(t^p, 1)$.
\end{example}

\section{Tropical curves with non transversal
intersection}\label{sec:nontransversal}

We fix two finite subsets $A_1, A_2$ of $\mathbb{Z}^2$ with $|A_1|, |A_2|\ge 2$
and such that $\Z A_1 + \Z A_2 = \Z^2$. In this section we define and study non
transversal intersections of two tropical curves associated to tropical
polynomials with respective supports $A_1,A_2$. As in
Definition~\ref{def:tropical_singularity}, we will say that the intersection is
non transverse when it comes from the tropicalization of a classical non
transverse intersection of two curves. In this case, we will see that the
standard definitions do not have a straightforward translation to the tropical
setting.

\begin{definition}\label{def:nonstransversal_intersection}
Let $f=\bigoplus_{i\in A_1} p^1_i \odot w^i$, $g= \bigoplus_{i\in A_2} p^2_i
\odot w^i$ be two tropical polynomials in $\R[w_1, w_2]$. Let $q$ be a point in
the intersection of the tropical curves $\mathcal{T}(f)\cap \mathcal{T}(g)$.
Then $q$ is a \emph{non-transversal (or non-smooth) intersection point} of
$\mathcal{T}(f)$ and $\mathcal{T}(g)$ if there exists two Laurent polynomials
$F =\sum_{i\in A_1} a^1_i x^i, G = \sum_{i\in A_2} a^2_i x^i$ in
$\mathbb{K}[x_1,x_2]$, with respective supports $A_1, A_2$ and a point $b\in
(\mathbb{K}^*)^2$ which is a non-transversal intersection of $\{F=0\}$ and
$\{G=0\}$ such that $Trop(F)=f$ (that is, $val(a^1_i)=p^1_i$ for all $i \in
A_1$), $Trop(G) = g$ and $val(b) =q$. 
\end{definition}

Recall that $b \in \{F=0\} \cap \{G=0\}$ is a non-transversal intersection point
if moreover the Jacobian $J_{F,G}$ vanishes at $b$. This Jacobian is the
determinant of the Jacobian matrix (or the matrix of the differential of the map
$(F,G)$), and also of its transpose $M_{F,G} =
\begin{pmatrix}F_{x_1}(b)&G_{x_1}(b)\\F_{x_2}(b)&G_{x_2}(b)\end{pmatrix}$. The
condition that $J_{F,G}(b) =0$ is obtained from elimination of variables from
the following equivalent fact: $\{F=0\}$ and $\{G=0\}$ intersect non
transversally at $b$ if and only if their tangent lines coincide, or
equivalently, the matrix $M_{F,G}$ has a non trivial kernel, that is, there
exists a non trivial vector $(y_1, y_2)$ which is a solution of the system
\begin{equation} \label{eq:y1y2}
F_{x_1}(b) y_1 + G_{x_1}(b) y_2 \, =\, F_{x_2}(b) y_1 + G_{x_2}(b) y_2 \, = \,0.
\end{equation}

Given two Laurent polynomials $F =\sum_{i\in A_1} a^1_i x^i, G = \sum_{i\in A_2}
a^2_i x^i$ in $\mathbb{K}[x_1,x_2]$, with respective supports $A_1, A_2 \subset
\Z^2$, the {\em mixed discriminant} of $F$ and $G$ is the $A$-discriminant
associated to the polynomial $y_1F+y _2G\in \mathbb{K}[x_1, x_2, y_1, y_2]$ with
support in the Cayley configuration (cf. \cite{GKZ-book}) $$e_1 \times A_1 \cup
e_2 \times A_2 \subset \Z^4.$$ In fact, this is a three dimensional
configuration lying in the plane defined by the sum of the two first coordinates
equal to $1$. This mixed discriminant vanishes at the vectors of coefficients
$((a^1_i)_{i \in A_1}, (a^2_i)_{i\in A_2})$ whenever $F$ and $G$ have a
non-transversal intersection at a point $b \in (K^*)^2$ for which the system
\eqref{eq:y1y2} has a solution $(y_1, y_2) \in (K^*)^2$. In particular, note
that horizontal and vertical tangents are not necessarily reflected in the mixed
discriminant (cf.~\cite[Section~3]{esterov-discriminant-projections} for a more
general definition of discriminants which takes into account different
supports). Hence, we do not take cover these extremal cases and will only
describe the non-transversal intersection points for which the system
\eqref{eq:y1y2} has a solution $(y_1, y_2) \in (K^*)^2$.

\begin{lemma}\label{lem:l}
Let $f,g \in \R[w_1, w_2]$ be two tropical bivariate polynomials with respective
supports $A_1, A_2$. The tropical plane curves that $f$ and $g$ define intersect
non-transversally at an intersection point $q= (q_1, q_2) \in \R^2$ if and only
if there exists $l\in \mathbb{R}$ such that $\overline{q}=(q_1,q_2,l)$ belongs
to the tropical discriminant associated to the polynomial $f\oplus w_3\odot g
\in \R[w_1, w_2,w_3]$ with support in the configuration ${\mathcal C}(A_1, A_2)
= A_1 \times \{0\} \cup A_2 \times \{1\} \subset \Z^3$, that is if
$\overline{q}$ is a singular point of $\mathcal{T}(f\oplus w_3\odot g)$. 
\end{lemma}

Note that the configurations ${\mathcal C}(A_1, A_2) \subset \Z^3 $ and the
Cayley configuration $e_1 \times A_1 \cup e_2 \times A_2 \subset \Z^4$ are
affinely equivalent. Therefore, the associated sparse discriminants coincide (up
to the names of the variables).

\begin{proof}
Suppose that there is an element $l\in \mathbb{R}$ such that $(q_1, q_2, l)$
belongs to the tropical discriminant of $f \oplus w_3 \odot g$. So, $(f \oplus
w_3 \odot g, (q_1, q_2, l))$ belongs to the incidence variety $Trop(H)$
associated to $A = \mathcal{C}(A_1, A_2)$. Then, there are algebraic
polynomials $F, G \in \K[x_1, x_2]$ and a point $(b_1, b_2, b_3) \in \K^3$ such
that $(F(x) + x_3 G(x),(b_1, b_2, b_3))$ lies in the incidence variety $H$,
$Trop(F)=f$, $Trop(G)=g$, $val(b_1, b_2, b_3)= (q_1, q_2, l)$. Hence $(b_1, b_2,
b_3)$ is a singular point of $F + x_3 G$, and so the partial derivatives must
vanish:
\[F_{x_1}(b_1, b_2) + b_3 G_{x_1}(b_1, b_2)=0,\ F_{x_2}(b_1, b_2) + b_3
G_{x_2}(b_1, b_2)=0, G(b_1, b_2)=0.\]
It follows that $F(b_1, b_2)=0$ and $b=(b_1, b_2)$ is a non-transversal
intersection point of $\{F=0\}$ and $\{G=0\}$. To prove the converse, suppose
$(q_1, q_2)$ is a non transversal intersection point of $f$ and $g$. There
exists $F,G \in \K[x_1, x_2]$ with respective supports $A_1, A_2$ such that
$Trop(F)=f, Trop(G) = g$ and $b=(b_1,b_2)$ as in
Definition~\ref{def:nonstransversal_intersection}. Then $F(b)=0$, $G(b)=0$, and
let $(y_1, y_2)$ be as in~\eqref{eq:y1y2}. It follows that $(b_1,b_2, y_2/y_1)$
is
a singular point of the surface defined by $F+ x_3G$, so $(F+x_3G,
(b_1,b_2,y_2/y_1))\in H$ and $(f\oplus w_3\odot g, (q_1,q_2, val(y_2/y_1))) \in
Trop(H)$. Therefore, $(q_1,q_2, val(y_2/y_1))$ is a non-transversal intersection
point of $f \oplus w_3 \odot g$.
\end{proof}

\begin{example} \label{ex:twolines}
Let $f=0\oplus 0 \odot w_1 \oplus 0 \odot w_2$ and $g=1 \oplus 0 \odot w_1
\oplus 0 \odot w_2$ be two tropical lines. These lines intersect at an infinite
number of points (all the points in the ray $\{w_1=w_2 \le 0\})$. However, two
algebraic lines intersect non-transversally if and only if they are the same,
and we expect that this also happens in the tropical setting. Since $f\neq g$,
let us check that they intersect transversally according to our definition.
Consider the surface defined by $f\oplus w_3\odot g= 0 \oplus 0 \odot w_1 \oplus
0 \odot w_2 \oplus 1 \odot w_3 \oplus 0 \odot w_1 \odot w_3 \oplus 0 \odot w_2
\odot w_3$. Both lines have the same support $A_1 = A_2 = \{ (0,0), (1,0),
(0,1)\}$. The associated mixed subdivision of ${\mathcal C}(A_1, A_2) =
\{(0,0,0), (1,0,0), (0,1,0), (0,0,1), (1,0,1), (0,1,1)\}$ is precisely the
marked subdivision occurring in Example~\ref{ex:nontriang2} and depicted in
Figure~\ref{fig:cayley2rectas}, so $f\oplus w_3\odot g$ is indeed non singular
according to Lemma~\ref{lem:l}.
\end{example}

We now revisit Example~\ref{ex:defective}.

\begin{example}
Let $A_1 = \{ (0,0), (1,0), (2,0)\}$, $A_2 = \{(0,0), (0,1), (0,2) \}$. So, $f$
and $g$ are in fact univariate polynomials in different variables. The
associated configuration ${\mathcal C}(A_1,A_2)$ is just the configuration $A$
occurring in Example~\ref{ex:defective}. In this case, the mixed discriminant
has codimension bigger than $1$ and a point of intersection $q= (q_1, q_2)$ is
non-transversal if and only if $q_1$ is a singular point of $\mathcal{T} (p_1
\oplus p_2 \odot w_1 \oplus p_3 \odot w_1^2)$ and $q_2$ is a singular point of
$\mathcal{T} (p_4 \oplus p_5 \odot w_2 \oplus p_6 \odot w_2^2)$, which is the
translation of the fact that $q$ is singular if and only if it lies in the
rowspan of $A$.
\end{example}

Since we are looking for the singular points $(q_1,q_2,l)$ in ${\mathcal
T}(f\oplus w_3\odot g)$, we could use the known tropical basis of the
discriminant of this surface to compute them, as in the discussion above.
However, we would like to obtain a method that involves only $f$ and $g$ and no
new variables. Just checking if the Jacobian matrix is singular will not work.

\begin{example}
Let $f=0\oplus 0 \odot w_1\oplus 0 \odot w_2\oplus 0\odot w_1 \odot w_2 \oplus 0
\odot w_2^2 \oplus 1 \odot w_1^2$, $g=0 \oplus 1 \odot w_1 \oplus 0 \odot w_2$.
Consider the intersection point $q=(0,0)$. If we would like to use a kind of
tropical Jacobian matrix, a natural choice would be the matrix
{\tiny{
\[\begin{pmatrix}\frac{\partial f}{\partial \{w_1=0\}} & \frac{\partial
f}{\partial \{w_2=0\}}\\
\frac{\partial g}{\partial \{w_1=0\}} & \frac{\partial g}{\partial
\{w_2=0\}}\end{pmatrix}
=
\begin{pmatrix} 0 \odot w_1 \oplus 0 \odot w_1 \odot w_2 \oplus 1 \odot w_1^2 &
0 \odot w_2\oplus 
0 \odot w_1 \odot w_2 \oplus 0 \odot w_2^2\\ 1 \odot w_1 &
0 \odot w_2 \end{pmatrix}.\]}}
When we evaluate at $q$, we get the matrix
$\begin{pmatrix}0& 0\\ 1 & 0\end{pmatrix}$
which is nonsingular because $0 \odot 0 \not = 0 \odot 1$.
However, take $F=(1-t) + 2x + 2y + 2 x y + y^{2} +x^{2} t$, $G=(1+t) + tx + y$,
$b=(-1,-1)$ is an intersection point and $Trop(F)=f$, $Trop(G)=g$,
$val(b)=(0,0)$ and
\[(1,2) \begin{pmatrix}F_x(b) & F_y(b) \\ G_x(b) & G_y(b)\end{pmatrix}=(1,2)
\begin{pmatrix} -2t& -2 \\ t& 1\end{pmatrix} = (0,0)\]
So the tropical point $(0,0)$ is a non-transversal intersection point.
The problem is that in general, given a polynomial $J \in \K[x_1, \dots, x_n]$
and a point $b \in (\K^*)^n$, $Trop(J) (val(b)) \not = val (J(b))$. In our case
$J$ is the Jacobian $J_{FG} = 2-2t+ y (2-2t)$ and $b=(-1,-1)$. We have that
$Trop(J_{FG})(val(b)) = (0 \oplus 0 \odot w_2)(0,0) = 0$, while $val(J_{FG})(b)
= \infty$. 
\end{example}

We now show another phenomenon that occurs.

\begin{example}
Let $F = G = 1 + x_1 + x_2$. $f =Trop(F)$ and $g=Trop(G)$ equal $0\oplus 0 \odot
w_1\oplus 0 \odot w_2$ and intersect non transversally at any point of
$\mathcal{T}(f) = \mathcal{T}(g)$. The toric Jacobian $x_1 x_2 J_{FG}$ is the
determinant of the $2 \times 2$ matrix $M$ with columns the Euler derivatives of
$F$ and $G$ with respect to $L=w_1, w_2$, and so it is identically zero. If we
consider the matrix $Trop(M)$ obtained by taking the tropicalization of each of
the entries of $M$, we get the $2 \times 2$ matrix with two equal rows $[0\odot
w_1 \, \, 0 \odot w_2]$. The tropical determinant of $Trop(M)$ equals the
permanent $(0 \odot w_1) \odot (0 \odot w_2) \oplus (0 \odot w_1) \odot (0 \odot
w_2)$. If we forget the fact that we have {\em twice} the term $(0 \odot w_1)
\odot (0 \odot w_2) = 0 \odot w_1 \odot w_2$, we lose information. We just get a
monomial, which defines an empty tropical curve.
\end{example}

We deal now with two easy cases.

\begin{proposition}\label{teo:non_transversal_intersection_two_vertices}
Let $f$, $g$ be two tropical bivariate polynomials with respective supports
$A_1, A_2$. Let $q=(q_1, q_2)$ be a tropical point that is a vertex in both
$\mathcal{T}(f)$ and $\mathcal{T}(g)$. Then $q$ is a non-transversal
intersection point of $f$ and $g$.
\end{proposition}
\begin{proof}
Set $l:=f(q)-g(q)\in \mathbb{R}$. We claim that $\overline{q}=(q_1,q_2,l)$ is a
singular point of the tropical surface $\mathcal{T}(f\oplus w_3\odot g)$. To see
this, note that the points in ${\mathcal C}(A_1, A_2)$ where the minimum in
$\overline{q}$ is attained are of the form $(i,0) \in A_1 \times 0$, for all $i$
such that $f(q)$ attains its minimum and those of the form $(i,1) \in A_2 \times
1$, for all $i$ where $g(q)$ attains its minimum. Since $q$ is a vertex of both
$\mathcal{T}(f)$ and $\mathcal{T}(g)$, the minimum of $f(q)$ is attained in a
$2$-dimensional cell $\sigma_1$ in $A_1$ and the minimum of $g(q)$ is
attained in a $2$-dimensional cell $\sigma_2$ in $A_2$. Let $L=j_1 w_1 +j_2 w_2
+j_3 w_3 +\b$ be any integer affine function in three variables, $L\neq
j_3w_3+\b$. Then $\{L=0\}$ cannot contain $\sigma_1$ nor $\sigma_2$, so there
are at least two different points of ${\mathcal C}(A_1, A_2)$ where the minimum
of $f\oplus w_3\odot g$ is attained at $\overline{q}$. It follows that
$\overline{q} \in \mathcal{T}\left( \frac{\partial (f\oplus w_3\odot
g)}{\partial L}\right)$. In case $L$ is of the form $j_3w_3+\b$, $\{L=0\}$ is
disjoint from ${\mathcal C}(A_1, A_2)$ unless $\b= -j_3-1,0$, and in this case
either all the monomials corresponding to the points in $A_1 \times 0$ or $ A_2
\times 1$ occur in the Euler derivative with respect to $L$ of $f+ w_3 \odot g$.
So, by Theorem~\ref{teo:singular_point_by_derivatives}, $\overline{q}$ is a
singular point of $f\oplus w_3\odot g$ and $q$ is a nontransversal intersection
point of $f$ and $g$.
\end{proof}

\begin{proposition}\label{
teo:non_transversal_intersection_one-vertex-and-segment}
Let $f,g$ be tropical polynomials with respective supports $A_1, A_2$. Suppose
that $q$ is a vertex of $\mathcal{T}(f)$ and lies on a segment or ray of
$\mathcal{T}(g)$. Moreover, suppose that, locally, $p$ is the only local
intersection point in $\mathcal{T}(f)\cap \mathcal{T}(g)$. Then $q$ is a
non-transversal intersection point of $f$ and $g$. 
\end{proposition}
\begin{proof}
The minimum of $f$ at $q$ is attained in at least three non-collinear monomials
associated to the points $\{\alpha_1, \alpha_2, \alpha_3\}$ and the minimum in
$g$ is attained in at least $2$ monomials associated to the points $\{\beta_1,
\beta_2\}$. Take again $l:= f(q) - g(q)$. Then, $(q_1,q_2, l)$ of $f\oplus w_3
\odot g$ is attained at the monomials corresponding to $C=\{(\alpha_1,0),
(\alpha_2,0), (\alpha_3,0), (\beta_1,1), (\beta_2,1)\}$. By construction,
$\{(\alpha_1,0), (\alpha_2,0), (\alpha_3,0)$, $(\beta_1,1)\}$ is a simplex.
Since the intersection of $\mathcal{T}(f)$ and $\mathcal{T}(q)$ around $q$ is
just the point $q$, it follows that the line through $(\beta_1,1)$ and
$(\beta_2,1)$ is not coplanar with any line generated by two different points
among $\{(\alpha_1,0), (\alpha_2,0), (\alpha_3,0)\}$. In other words, $C$ is a
circuit of dimension $3$. By Theorem~\ref{teo:char_punto_singular}, $(q_1, q_2,
l)$ is a singular point of $f\oplus w_3\odot g$ and thus $q$ is a
non-transversal intersection point of $f$ and $g$. 
\end{proof}

In order to deal with the general case, we introduce some notation. Let $L' =
j_1 w_1 + j_2 w_2 +c w_3 + \b$ be an integer affine function in $3$ variables.
Call $L= j_1w_1 +j_2w_2+\b \in \R[w_2, w_2]$, so that 
\begin{equation}\label{eq:L'}
L'=L+cw_3.
\end{equation}
We clearly have:
\begin{equation} \label{eq:partial}
\frac{\partial (f\oplus w_3\odot g)}{\partial L'}=\frac{\partial f}{\partial
L} \oplus w_3\odot \frac{\partial g}{\partial (L+c)}. 
\end{equation}

\begin{lemma}\label{lem:4types}
Let $q \in \R^2$ be an intersection point of $\mathcal{T}(f)$ and ${\mathcal
T}(g)$ and let $L'$ be an integer affine function in $3$ variables. There exists
$l \in \R$ such that $\overline{q}=(q_1,q_2,l)$ is a point in the tropical
surface $\mathcal{T}(\frac{\partial (f\oplus w_3 \odot g)}{\partial L'})$.
\end{lemma}

\begin{proof}
Define $L$ as in \eqref{eq:L'}. It is enough to take $l = \frac{\partial
f}{\partial L}(q)-\frac{\partial g}{\partial (L+c)}(q)$.
\end{proof}

In fact, keeping the notations of the previous Lemma, we can make a finer
classification of the affine linear forms $L'$ into $4$ types and describe all
possible choices of $l$ in each case. We say that $L'$ is of type 1 if the
minimum of $\frac{\partial f}{\partial L}$ at $q$ is attained at least twice and
exactly once in $\frac{\partial g}{\partial (L+c)}$. In this case, we can take
any $l \ge \frac{\partial f}{\partial L}(q)-\frac{\partial g}{\partial
(L+c)}(q)$. We say that $L'$ is of type 2 if instead the minimum at $q$ is
attained at least twice in $\frac{\partial g}{\partial L +c}$ and once in
$\frac{\partial f}{\partial L}$. In this case, we can take any $l \le
\frac{\partial f}{\partial L}(q)-\frac{\partial g}{\partial (L+c)}(q)$. The
integer affine linear function $L'$ is said to be of type 3 if the minimum is
attained at least twice in both $\frac{\partial f}{\partial L}$ and
$\frac{\partial g}{\partial (L+c)}$. In this case, we can take any $l\in \R$.
Finally, we say that $L'$ is of type 4 if the minimum is attained once in both
$\frac{\partial f}{\partial L}$ and $\frac{\partial g}{\partial (L+c)}$. In this
case, the only possible choice is $l = \frac{\partial f}{\partial
L}(q)-\frac{\partial g}{\partial (L+c)}(q)$.

We now present the main result in this section.

\begin{theorem}\label{teo:non_transversal_intersection_nontrivial_case}
Let $f$, $g$ be two tropical bivariate polynomials with respective supports
$A_1, A_2$ and $q$ an intersection point of $\mathcal{T}(f)$ and ${\mathcal
T}(g)$. Consider the set ${\mathcal A}(q) =\{(L_1, L_2, c_1, c_2)\}$ of all
$4$-tuples where $L_1, L_2$ are integer affine functions in two variables and
$c_1,c_2\in \mathbb{R}$, for which $q$ belongs to at {\em least one} of the
tropical curves associated to $\frac{\partial f}{\partial L_1}$, $\frac{\partial
f}{\partial L_2}$, $\frac{\partial g}{\partial (L_1+c_1)}$, $\frac{\partial
g}{\partial (L_2+c_2)}$. Let ${\mathcal B}(q)=\{(L_1,L_2,c_1,c_2)\}$ be the set
of
$4$-tuples for which $q$ does {\em not} belong to any of these four tropical
curves. 

Then, $q$ is a non-transversal intersection point of $f$ and $g$ if and only if
\begin{itemize}
\item For all $(L_1,L_2,c_1,c_2)\in \mathcal{A}(q)$, $q$ belongs to the
tropical curve associated to the tropical polynomial
\[\frac{\partial f}{\partial L_1}\odot \frac{\partial{g}}{\partial (L_2+c_2)}
\oplus \frac{\partial f}{\partial L_2}\odot \frac{\partial g}{\partial
(L_1+c_1)}.\]
\item The following equalities hold for all $(L_1,L_2,c_1,c_2) \in
\mathcal{B}(q)$:
\[\frac{\partial f}{\partial L_1}\odot \frac{\partial{g}}{\partial (L_2+c_2)}(q)
=
\frac{\partial f}{\partial L_2}\odot \frac{\partial g}{\partial (L_1+c_1)}(q).\]
\end{itemize}
It is enough to check these conditions for a finite number of $4$-tuples.
\end{theorem}

\begin{proof}
By Theorem~\ref{teo:singular_point_by_derivatives}, a point $q \in \R^2$ is a
non-transversal intersection point if and only if there is an $l$ such that
$(q_1, q_2, l)$ belongs to all the partial derivatives $\frac{\partial (f \oplus
w_3\odot g)}{\partial L'}$. Assume first that $q\in \R^2$ is a non-transversal
intersection point and let $l\in \mathbb{R}$ such that $(q_1, q_2, l)$ belongs
to all the partial derivatives $\frac{\partial (f \oplus w_3\odot g)}{\partial
L}$. Let $L_1, L_2$ be integer affine functions of $2$ variables and $c_1, c_2
\in \R$. Call $L'_1 = L_1 + c_1 w_3, L'_2 = L_2 + c_2 w_3$. Then,
by~\eqref{eq:partial}, we have that the minimum at $ \frac{\partial f}{\partial
L_i}(q) \oplus l \odot \frac{\partial g}{\partial (L_i+c_i)}(q)$ is attained
twice for $i=1,2$. It is not difficult to see that the two conditions in the
statement of the Theorem hold, separating the arguments for the different types
of functions $L'$. Reciprocally, assume that these conditions hold for all
$4$-tuples of the form $(L_1, L_2, c_1, c_2)$ and let $q$ be an intersection
point. Then, any $(q_1, q_2, l)$ belongs to the tropical surfaces defined by
$f\oplus w_3\odot g$ and $\frac{\partial (f\oplus w_3\odot g)}{\partial (w_3 +
\b)}$ for all constants $\b$. Hence, it remains to check the Euler derivatives
when $L'$ for which $L \not=0$.

Let $l_1$ be the maximum of $\frac{\partial f}{\partial L}(q) - \frac{\partial
g}{\partial (L+c)}(q)$ for all $L'=L+w_3c$ of type 1. If there are no $L'$ of
type 1 set $l_1=-\infty$. Let $l_2$ be the minimum of $\frac{\partial
f}{\partial L}(q) - \frac{\partial g}{\partial (L+c)}(q)$ for all $L'=L+w_3c$ of
type 2. If there are no $L'$ of type 2, set $l_2=\infty$. Finally, let $L'_1,
\ldots, L'_r$ be the integral affine functions of type 4 whose supports
represent all possible supports with this type, and write $L'_i=L_i+c_i$. Call
$l_{4}(L'_i) = \frac{\partial f}{\partial L_i}(q) - \frac{\partial g}{\partial
(L_i+c_i)}(q)$. Then, it is clear that there exists an $l$ such that $(q_1, q_2,
l)$ belongs to all partial derivatives if and only if \[l_1\leq l_{4}(L'_1) =
l_{4}(L'_2) = \ldots = l _{4}(L'_r) \leq l_2\] We have to translate these
conditions into Jacobian like equations.

Let $L_1'=L_1+w_3c_1$, $L_2'=L_2 + w_3c_2$ be two integral affine linear
functions. If $L_1$ and $L_2$ are of the same type 1 or 2, or if one of the
types is 3, then it always happens that $q$ is a point in the curve
$\frac{\partial f}{\partial L_1}\odot \frac{\partial g}{\partial (L_2+c_2)}(q)
\oplus \frac{\partial f}{\partial L_2}\odot \frac{\partial g}{\partial
(L_1+c_1)}(q)$, because we are adding two polynomials for which $q$ is a
``zero''. If $L'_1$ is of type $1$ and $L'_2$ is of type 2 then: $\frac{\partial
f}{\partial L_1}$ and $\frac{\partial g}{\partial (L_2+c_2)}$ attain its minimum
at least twice in $q$, and $\frac{\partial f}{\partial L_2}$ and $\frac{\partial
g}{\partial (L_1+c_1)}$ attains its minimum once in $q$. Hence $q$ is a point in
$\frac{\partial f}{\partial L_1}\odot\frac{\partial g}{\partial (L_2+c_2)}
\oplus \frac{\partial f}{\partial L_2}\odot\frac{\partial g}{\partial
(L_1+c_1)}$ if and only if $\frac{\partial f}{\partial L_1}\odot\frac{\partial
g}{\partial (L_2+c_2)}(q) \leq \frac{\partial f}{\partial
L_2}\odot\frac{\partial g}{\partial (L_1+c_1)}(q)$. This is equivalent to
$\frac{\partial f}{\partial L_1}(q) + \frac{\partial g}{\partial (L_2+c_2)}(q)
\leq \frac{\partial f}{\partial L_2}(q) + \frac{\partial g}{\partial
(L_1+c_1)}(q)$. So, $\frac{\partial f}{\partial L_1}(q) - \frac{\partial
g}{\partial (L_1+c_1)}(q) \leq \frac{\partial f}{\partial L_2}(q) -
\frac{\partial g}{\partial (L_2+c_2)}(q)$. This happens for every pair of
integral affine functions $L_1$ of type 1 and $L_2$ of type 2 if and only if
$l_1\leq l_2$. Let $L'_1$ be of type 1 and $L'_2$ of type 4. Then $q$ is a point
in $\frac{\partial f}{\partial L_1}\odot\frac{\partial g}{\partial (L_2+c_2)}
\oplus \frac{\partial f}{\partial L_2}\odot\frac{\partial g}{\partial
(L_1+c_1)}$ for all $L'_1$ of type 1 if and only if $l_1\leq l_4(L'_2)$. Let
$L'_1$ be of type 2 and $L'_2$ of type 4. Then $q$ is a point in $\frac{\partial
f}{\partial L_1}\odot\frac{\partial g}{\partial (L_2+c_2)} \oplus \frac{\partial
f}{\partial L_2}\odot\frac{\partial g}{\partial (L_1+c_1)}$ for all $L'_1$ of
type 2 if and only if $l_4(L'_2) \leq l_2$. Finally, if both equations are of
type 4, then the condition $l_4(L_1')=l_4(L_2')$ translates into:
$\frac{\partial f}{\partial L_1}(q) - \frac{\partial g}{\partial (L_1+c_1)}(q) =
\frac{\partial f}{\partial L_2}(q) - \frac{\partial g}{\partial (L_2+c_2)}(q)$.
We deduce that $\frac{\partial f}{\partial L_1}\odot \frac{\partial g}{\partial
(L_2+c_2)}(q) = \frac{\partial f}{\partial L_2}\odot \frac{\partial g}{\partial
(L_1+c_1)}(q)$, which ends the proof.
\end{proof}

Note that in the last equality of the above proof, there is only one monomial on
each side of the equality where the minimum is attained. If this monomial
happens to be the same on both sides, we cannot ensure that $q$ is on the
variety $\frac{\partial f}{\partial L_1}\odot \frac{\partial g}{\partial
(L_2+c_2)}\oplus \frac{\partial f}{\partial L_2}\odot \frac{\partial g}{\partial
(L_1+c_1)}$. Both items in the statement of
Theorem~\ref{teo:non_transversal_intersection_nontrivial_case} are similar but
the first one concerns tropical varieties, while the second one needs to deal
with equalities of the values taken by two tropical polynomials, that do not
represent tropical varieties. A more homogeneous approach can be given if we use
the supertropical algebra introduced by Izhakian and then developed by Izhakian
and Rowen, to unify both conditions.
 
Recall that one can consider an {\em extended tropical semiring} $(\T', \oplus',
\odot')$ constructed from our tropical ring $(\T, \oplus, \odot)$
\cite{tropical-arithmetic-Zur}. This semiring structure has a partial idempotent
addition that distinguishes between sums of similar elements and sums of
different elements. Set theoretically, this bigger semiring $\T'$ is composed
from the disjoint union of two copies of $\R$, denoted $\R$ and $\R^\nu$, plus
the neutral element for the sum $\infty$. There is a natural bijection $\nu: \R
\to \R^\nu$ and the operation $\oplus'$ verifies $a \oplus' a = \nu(a)$, for all
$a \in \R$. The elements in $\R^\nu$ are called \emph{ghosts}, so $a {\oplus}'a$
is a ghost element. This terminology reflects the idea that in a field $\K$ with
a valuation, if two elements have the same valuation, we cannot predict in
general the valuation of their sum. We refer to
\cite{tropical-arithmetic-Zur,Supertropical-algebra} for further details, in
particular for the full definition of the operations in this supertropical
algebra.

An element $q \in \T'^d$ lies in the variety $\mathcal{T}'(h)$ defined by a
supertropical polynomial $h = \bigoplus'_{i\in A} p_i \odot' w^i$ with a
finite support set $A \in \Z^d$, when $h(q)$ is a ghost element. Given $A_1,
A_2$ as before and two supertropical polynomials $f, g$ with respective supports
in $A_1, A_2$ and coefficients in $\R$ (the so called tangible elements of
$\T'$), we can mimic our previous definitions. Thus, a point $q \in \R^2$ is
said to be a non-transversal intersection of $f$ and $g$ if there exists $l \in
\R$ such that $(\frac{\partial (f \oplus' w_3 \odot'g)}{\partial L'} )(q_1,
q_2, l)$ is a ghost element for any integer affine linear form $L'$ in $3$
variables.

We can translate Theorem~\ref{teo:non_transversal_intersection_nontrivial_case}
in the following terms:

\begin{theorem}\label{teo:non_transversal_intersection_supertropical}
Let $f, g$ be two supertropical bivariate polynomials with respective supports
$A_1, A_2$ and let $q \in \mathcal{T}'(f) \cap \mathcal{T}'(g)$. Then $q$ is a
non-transversal intersection point if and only if for all $4$-tuples $(L_1, L_2,
c_1, c_2)$, it holds that $q$ lies in the variety defined by the supertropical
polynomial $\frac{\partial f}{\partial L_1}{\odot}' \frac{\partial{g}}{\partial
(L_2+c_2)}$ ${\oplus}' \frac{\partial f}{\partial L_2}{\odot}' \frac{\partial
g}{\partial (L_1+c_1)}$; that is, the value of this polynomial at $q$ is a
ghost.
\end{theorem}

\noindent {\bf Acknowledgments:} This work started during our stays at the
Mathematical Sciences Research Institute (MSRI) during the Special Semester on
Tropical Geometry, where we enjoyed a wonderful working atmosphere. AD is
grateful for the support of the Simons Foundation during that period. We thank
Cristian Czubara for useful remarks and Antonio Laface for stimulating
questions. We also thank the referees for insightful remarks.


\begin{thebibliography}{10}

\bibitem{Bergman-complex}
Federico Ardila and Caroline~J. Klivans.
\newblock The {B}ergman complex of a matroid and phylogenetic trees.
\newblock {\em J. Combin. Theory Ser. B}, 96(1):38--49, 2006.

\bibitem{Computing_trop_var}
T.~Bogart, A.~N. Jensen, D.~Speyer, B.~Sturmfels, and R.~R. Thomas.
\newblock Computing tropical varieties.
\newblock {\em J. Symbolic Comput.}, 42(1-2):54--73, 2007.

\bibitem{inflection-points}
E.~Brugall\'e and L.~L\'opez de Medrano.
\newblock Inflection points of real and tropical plane curves.
\newblock {arXiv:1102.2478}, 2011.

\bibitem{tropical-discriminant}
Alicia Dickenstein, Eva~Maria Feichtner, and Bernd Sturmfels.
\newblock Tropical discriminants.
\newblock {\em J. Amer. Math. Soc.}, 20(4):1111--1133 (electronic), 2007.

\bibitem{kapranov}
Manfred Einsiedler, Mikhail Kapranov, and Douglas Lind.
\newblock Non-{A}r\-chi\-me\-de\-an amoebas and tropical varieties.
\newblock {\em J. Reine Angew. Math.}, 601:139--157, 2006.

\bibitem{esterov-discriminant-projections}
A.~Esterov.
\newblock Newton polyhedra of discriminants of projections.
\newblock {\em Discrete Comput. Geom.}, 44(1):96--148, 2010.

\bibitem{GKZ-book}
I.~M. Gel'fand, M.~M. Kapranov, and A.~V. Zelevinsky.
\newblock {\em Discriminants, resultants, and multidimensional determinants}.
\newblock Mathematics: Theory \& Applications. Birkh\"auser Boston Inc.,
 Boston, MA, 1994.

\bibitem{tropical-arithmetic-Zur}
Zur Izhakian.
\newblock Tropical arithmetic and matrix algebra.
\newblock {\em Comm. Algebra}, 37(4):1445--1468, 2009.

\bibitem{Supertropical-algebra}
Zur Izhakian and Louis Rowen.
\newblock Supertropical algebra.
\newblock {arXiv:0806.1175.}, 2008.

\bibitem{sing-fixed-point}
H.~Markwig, T.~Markwig, and E.~Shustin.
\newblock Tropical curves with a singularity in a fixed point.
\newblock {\em Preprint, arXiv:0909.1827}, 2009.

\bibitem{mikha}
G.~Mikhalkin.
\newblock Tropical geometry, draft of the book in preparation.
\newblock {\em Available at: http://www.math.toronto.edu/mikha/book.pdf}.

\bibitem{Mik05}
Grigory Mikhalkin.
\newblock Enumerative tropical algebraic geometry in {$\mathbb{R}\sp 2$}.
\newblock {\em J. Amer. Math. Soc.}, 18(2):313--377 (electronic), 2005.

\bibitem{Tesis-master-Ochse}
Dennis Ochse.
\newblock The relation between the tropical A-discriminant and the secondary
 fan, 2009.

\bibitem{stubook}
Bernd Sturmfels.
\newblock {\em Solving systems of polynomial equations}, volume~97 of {\em CBMS
 Regional Conference Series in Mathematics}.
\newblock Published for the Conference Board of the Mathematical Sciences,
 Washington, DC, 2002.

\bibitem{Transfer-trop}
Luis~Felipe Tabera.
\newblock Tropical plane geometric constructions: a transfer technique in
 tropical geometry.
\newblock {\em Rev. Mat. Iberoamericana}, 27(1):181--232, 2011.

\end{thebibliography}
\end{document}